\documentclass{amsart}

\usepackage[foot]{amsaddr}
\usepackage{amsmath}
\usepackage{amsthm}
\usepackage{amssymb}
\usepackage{tikz}
\usetikzlibrary{automata,positioning}
\usepackage{enumitem}
\usepackage{mathrsfs}
\usepackage{scalerel}

\usepackage{hyperref}

\newtheorem{theorem}{Theorem}[section]

\newtheorem{alphatheorem}{Theorem}

\theoremstyle{definition}

\newtheorem{definition}[theorem]{Definition}
\newtheorem{proposition}[theorem]{Proposition}

\newtheorem*{proposition*}{Proposition}
\newtheorem*{observation*}{Observation}
\newtheorem*{claim*}{Claim}

\newtheorem*{lemma*}{Lemma}
\newtheorem{example}[theorem]{Example}
\newtheorem{corollary}[theorem]{Corollary}
\newtheorem{remark}[theorem]{Remark}
\newtheorem*{remark*}{Remark}

\newtheorem*{conjecture*}{Conjecture}

\newtheorem*{convention*}{Convention}

\theoremstyle{plain}
\newtheorem{lemma}[theorem]{Lemma}

\newcommand{\bra}[1]{ \left( #1 \right) }

\renewcommand{\tilde}{\widetilde}

\newcommand{\abs}[1]{\left|#1\right|}
\newcommand{\fp}[1]{\left\{ #1 \right\}}

\newcommand{\norm}[1]{\left\lVert #1 \right\rVert}

\renewcommand{\b}{\beta}

\newcommand{\cN}{\mathcal{N}}

\newcommand{\fpa}[1]{\left\lVert #1 \right\rVert_{\mathbb{R}/\mathbb{Z}}}
\newcommand{\e}{\varepsilon}
\renewcommand{\a}{\alpha}
\renewcommand{\b}{\beta}

\usepackage{mathrsfs}
\usepackage[mathscr]{euscript}
	
\renewcommand{\b}{\beta}

\newcommand{\NN}{\mathbb{N}}
\newcommand{\QQ}{\mathbb{Q}}

\newcommand{\PP}{\mathbb{P}}

\DeclareMathOperator*{\EE}{\scalerel*{\mathbb{E}}{\textstyle\sum}}

\newcommand{\ZZ}{\mathbb{Z}}
\newcommand{\RR}{\mathbb{R}}

\newcommand{\CC}{\mathbb{C}}

\newcommand{\TT}{\mathbb{T}}
\newcommand{\cA}{\mathcal{A}}
\newcommand{\cB}{\mathcal{B}}

\newcommand{\cM}{\mathcal{M}}

\newcommand{\floor}[1]{\left\lfloor #1 \right\rfloor}

\newcommand{\la}{\langle}
\newcommand{\ra}{\rangle}

\newcommand{\set}[2]{\left\{ #1 \ \middle| \ #2 \right\} }

\newcommand{\Fix}{\mathrm{Fix}}

\newcommand{\aveN}{\frac1N\sum_{n=1}^N}

\newcommand{\parbreak}[1]{
\begin{center}
***
\end{center}
}

\usepackage{color}

\usepackage{soul}

\usepackage{stmaryrd}

\newcommand{\ceil}[1]{\left\lceil #1 \right\rceil}

\title[Automatic sequences as good weights]{Automatic sequences \\as good weights for ergodic theorems}
\date{}

\author[T.\ Eisner]{Tanja Eisner}
\address[T.\ Eisner]{Institute of Mathematics, University of Leipzig, P.O. Box 100 920, 04009 Leipzig, Germany}
\email{eisner@math.uni-leipzig.de}
\author[J.\ Konieczny]{Jakub Konieczny}
\address[J.\ Konieczny]{Einstein Institute of Mathematics Edmond J. Safra Campus, The Hebrew University of Jerusalem Givat Ram. Jerusalem, 9190401, Israel}
\address{Faculty of Mathematics and Computer Science, Jagiellonian University in Krak\'{o}w, \L{}ojasiewicza 6, 30-348 Krak\'{o}w, Poland}
\email{jakub.konieczny@gmail.com}


\begin{document}

\maketitle 

\begin{abstract}

We study correlation estimates of automatic sequences (that is, sequences computable by finite automata) with polynomial phases. As a consequence, we provide a new class of good weights for classical and polynomial ergodic theorems. 
We show that automatic sequences are good weights in $L^2$ for polynomial averages and totally ergodic systems. For totally balanced automatic sequences (i.e., sequences converging to zero in mean along arithmetic progressions) the pointwise weighted ergodic theorem in $L^1$ holds. Moreover, invertible automatic sequences are good weights for the pointwise polynomial ergodic theorem in $L^r$, $r>1$. 
\end{abstract}

\makeatletter{}\section{Introduction}

The study of weighted ergodic averages goes back to Wiener and Wintner \cite{WienerWintner} who showed that sequences of the form $(\lambda^n)$ for $\lambda$ in the unit circle $\TT$ form a set of good weights for the pointwise ergodic theorem in $L^\infty$ with the set of convergence being independent of $\lambda$. Recall that a sequence $(a_n)\subset\CC$ is called \emph{a good weight for the pointwise ergodic theorem in $L^p$} if for every measure-preserving system $(X,\mu,T)$ and every $f\in L^p(\mu)$ the weighted averages 

$$
\aveN a_nT^nf
$$
converge a.e.. Note that some authors call the above averages ``modulated'' instead of ``weighted'', see, e.g., Berend,  Lin, Rosenblatt, Tempelman \cite{BLRT}. 

Since then quite a few classes of good weights have been discovered. The celebrated return time theorem of Bourgain \cite{BFKO} states that for an ergodic measure-preserving system $(X,\mu,T)$, every $f\in L^\infty(\mu)$ and a.e.\ $x\in X$, the sequence $(f(T^nx))$ is a good weight in $L^\infty$ (and hence in $L^1$ by the maximal inequality). Lesigne \cite{Lesigne90,Lesigne93} has extended the linear weights $(\lambda^n)$ in the Wiener-Wintner result to polynomial ones $(\lambda^{p(n)})$, see also Frantzikinakis \cite{Frantz} for the study of uniform convergence, and the general case of so-called nilsequences was treated in Host, Kra \cite{HostKra09} with the uniform convergence version in Eisner, Zorin-Kranich \cite{EisnerZorin}. Thus, for good systems (in this case nilsystems), Bourgain's return times result holds everywhere instead of almost everywhere. 

Until now there is a limited number of different examples of good weights known.
Such examples are the von Mangoldt and the M\"obius function treated in Wierdl \cite{Wierdl88} and  El Abdalaoui, Ku\l{}aga-Przymus, Lema\'{n}czyk, de la Rue \cite{AHKLR}, respectively, $q$-multiplicative sequences as in Lesigne, Mauduit, Moss\'{e} \cite{LesigneMauduitMosse-1994,LesigneMauduit-1996}, and
 Hardy fields weights studied in Eisner, Krause \cite{EisnerKrause}, see also Krause, Zorin-Kranich \cite{KrauseZorin} for a random version. Note that the argument in \cite{AHKLR} also shows that the M\"obius function times a nilsequence is also a good weight. 

For more results on convergence of weighted ergodic averages see, e.g., Bellow, Losert \cite{BellowLosert85}, Assani  \cite{Assani-book,Assani98}, {C}\"omez, Lin, Olsen \cite{CLO}, Lin, Olsen, Tempelman \cite{LOT}, Berend,  Lin, Rosenblatt, Tempelman \cite{BLRT}, Host, Kra \cite{HostKra09}, Chu \cite{Chu09}, Assani, Presser \cite{AssaniPresser}, Assani, Moore \cite{AssaniMoore}, Zorin-Kranich \cite{Zorin15}, Eisner \cite{Eisner13,Eisner15}, see also e.g. Cuny, Weber \cite{CunyWeber} for related results.

The purpose of this paper is to give a new class of examples of good weights for the pointwise ergodic theorem and its polynomial version,  namely automatic sequences. 

\mbox{}\\ 

Automatic sequences are simply the sequences computable by finite automata (see Section \ref{sec:Definitions} for a precise definition). They are of considerable interest in computer science, as they give rise to one of the weakest notions of computability. For extensive background, see \cite{AlloucheShallit-book}. 

Computations of exponential sums involving automatic sequences are a standard tool, often used to solve number theoretic problems. Given how ubiquitous the many variants of the circle method are in modern number theory, this comes as no surprise. For a good source of background and discussion, see \cite{Mullner-thesis}.

Before we move on with the discussion, let us note that in the simplest instance, one may consider the average value $\EE_{n<N} a(n)$ (for notation, see the end of this section), where $a(n)$ is an automatic sequence. Unfortunately, these do not converge in general as $N \to \infty$, although the logarithmic averages $\frac{1}{\log N}\sum_{n<N} a(n)/(n+1)$ do. Some of the technical complications in this paper can be traced back to this kind of behaviour.

Perhaps the simplest non-trivial result in this vein is due to Gelfond \cite{Gelfond-1968}, who showed that for the Thue--Morse sequence $(t(n))$ it holds that $\abs{\EE_{n<N} t(n) e(n \alpha)} \ll N^{-c}$ uniformly in $\alpha$, and gave the optimal value of $c$. Here, $t$ is given by $t(n) = (-1)^{s_2(n)}$, where $s_2(n)$ denotes the sum of binary digits of $n$. For similar results involving the Rudin--Shapiro sequence, see \cite{MauduitSarkozy-1998} and references therein. The Rudin--Shapiro sequence is given by $r(n) = -1$ if the number of times the pattern $11$ appears in binary expansion of $n$ is odd, and $r(n) = +1$ if the said number is even. In \cite{HostKra09} it is shown that the Thue--Morse sequence is a good sequence of weights for mean convergence of multiple ergodic averages. In \cite{Konieczny-2017+}, Konieczny obtained bounds on the Gowers norms of the Thue--Morse and Rudin--Shapiro sequences, which imply that these sequences do not correlate with any polynomial phases.

For the purposes of this paper, we will study correlations of fairly general classes of automatic sequences with linear phases (Prop.\ \ref{prop:cancellation-qualitative}, \ref{prop:cancellation-quantitative-A}) and polynomial phases (Cor.\ \ref{prop:cancellation-qualitative-polynomial}, Prop.\ \ref{prop:cancel-quant-poly}). Similar results for Kloosterman sums are obtained in \cite{DrappeauMullner-2017+}.
Related work for infinite automata can be found in \cite{Mauduit-2006}. 

Another interesting result related to sums involving automatic sequences is due to M\"{u}llner \cite{Mullner-2017+}, who proved that the Sarnak Conjecture holds for automatic sequences. In particular, if $a(n)$ is an automatic sequence then 
$
	\EE_{n<N} a(n) \mu(n) \to 0
$
as $N \to \infty$, where $\mu$ denotes the M\"{o}bius function. For related results see also \cite{MartinMauduitRivat-2014}.

\mbox{}\\

The first of our main results shows that automatic sequences are good weights for $L^2$-convergence for totally ergodic systems. Similar results for a related class of $q$-multiplicative sequences are obtained in \cite{LesigneMauduitMosse-1994}, \cite{LesigneMauduit-1996}.

\begin{alphatheorem}\label{thm:A}
Let $a\colon \NN \to \CC$ be an automatic sequence and let $p \in \ZZ[x]$ be a nonconstant polynomial with $p(\NN_0) \subset \NN_0$. Then, for any totally ergodic measure-preserving system $(X,\mu, T)$ and any {$f \in L^1(\mu)$} with $\int_X f d\mu = 0$ we have

$$
	\EE_{n < N} a(n) T^{p(n)} f \to 0 \text{ in } L^2 \text{ as } N \to \infty.
$$
\end{alphatheorem}
\begin{remark*}
	Using the above result as a black-box and passing to arithmetic progressions, one can derive a slightly more general analogue, where the condition  $p \in \ZZ[x]$ is replaced with $p \in \QQ[x]$ (and the condition $p(\NN_0) \subset \NN_0$ is kept unchanged). This covers polynomials such as $p(x) = x(x+1)/2$. 
\end{remark*}

For almost everywhere convergence, we need to impose some mild conditions on the automatic sequence. Let us say that a (bounded) sequence $a\colon \NN_0 \to \CC$ is \emph{balanced} if $\EE_{n<N} a(n) \to 0$ as $N \to \infty$, and \emph{totally balanced} if $\EE_{n < N} a(qn+r) \to 0$ as $N \to \infty$ for any $q \in \NN,\ r \in \NN_0$.  Equivalently, $a(n)$ is totally balanced if it does not correlate with any periodic sequence $b(n)$: $ \EE_{n < N} a(n) b(n) \to 0$ as $N \to \infty$. For instance, $\bra{ (-1)^n}$ is balanced but not totally balanced, while $(t(n))$ is totally balanced.

\begin{alphatheorem}\label{thm:B}
	Let $a\colon \NN \to \CC$ be a totally balanced automatic sequence. Then, for any ergodic measure-preserving system $(X,\mu, T)$ and any {$f \in L^1(\mu)$} we have
	
$$
	\EE_{n < N} a(n) T^n f(x) \to 0 \text{ for a.e. } x \in X \text{ as } N \to \infty.
$$
\end{alphatheorem}

\begin{remark*}
Note that since automatic sequences are bounded, it suffices to show a.e.\ convergence for $L^2$-functions in both theorems by the classical maximal inequality.

For polynomial averages, by Bourgain's maximal inequality  for polynomials \cite{Bourgain89}, a.e.~convergence in $L^2$ implies a.e.~convergence in $L^p$ for every $p>1$. Note that even unweighted monomial averages diverge in $L^1$ (and that monomials have even a much stronger property of being universally bad in $L^1$) by Buczolich, Mauldin \cite{BuczolichMauldin} and LaVictoire \cite{LaVictoire}.
\end{remark*}

Theorem \ref{thm:B} holds (with a natural modification of the limit) also when $a(n)$ is a sum of a totally balanced sequence and a periodic sequence.
Unfortunately, not every $k$-automatic sequence admits a decomposition into a periodic and totally balanced part, as is seen from the example of the sequence $a(n) = (-1)^{\nu_2(n)}$ (where $\nu_2(n)$ denotes the largest power of $2$ dividing $n$). However, \emph{invertible} automatic sequences (see Section \ref{sec:Invertible} for details) admit such a decomposition. In fact, for invertible sequences we obtain a considerably stronger conclusion.

\begin{alphatheorem}\label{thm:D}
	Let $a\colon \NN_0 \to \CC$ be an invertible automatic sequence {and let $p \in \ZZ[x]$ be a  
polynomial with $p(\NN_0) \subset \NN_0$}. Then, for any ergodic measure-preserving system $(X,\mu, T)$ and any {$f \in L^r(\mu)$, $r>1$,} {the averages

$$
	\EE_{n < N} a(n) T^{p(n)} f(x) $$
converge a.e.\ as $N \to \infty$. If $p$ is linear then the convergence holds for any $f \in L^1(\mu)$.}
\end{alphatheorem}

\subsection*{Notation}

We denote $\NN = \{1,2,\dots,\}$ and $\NN_0 = \NN \cup \{0\}$. The symbol $\EE$ is borrowed from probability theory, { $\EE_{x \in A} f(x) = \frac{1}{\abs{A}} \sum_{x \in A} f(x)$ for a finite set $A$}. We write $[N] = \{0,1,\dots,N-1\}$ and $e(\theta) = e^{2 \pi i \theta}$.

We use standard asymptotic notation: $X = O(Y)$ or $X \ll Y$ if there exists an absolute constant $c$ such that $\abs{X} < c Y$. If $X$ and $Y$ depend on a parameter $n$ then $X = o(Y)$ as $n \to \infty$ if $Y > 0$ for sufficiently large $n$ and $X/Y \to 0$ as $n \to \infty$.

\subsection*{Acknowledgements}

The authors wish to thank Jakub Byszewski and Emmanuel Lesigne for helpful discussions and comments, and the anonymous referees for their careful reading of this paper. 
The second author also expresses his gratitude to the University of Oxford and to the University of Science and Technology of China in Hefei, where parts of this paper were completed. The second author is supported by ERC grant ErgComNum 682150.

\makeatletter{}\section{Definitions}\label{sec:Definitions}

\subsection*{Automatic sequences}

A sequence $(a(n))_{n \geq 0}$ taking values in a finite set $\Delta$ is $k$-\textit{automatic} if $a(n)$ can be computed by a finite device, given the expansion of $n$ base $k$ on input. We now make this more precise. For the canonical introduction to the theory of automatic sequences, we refer to \cite{AlloucheShallit-book}.

Let $k \geq 2$ be an integer. We will denote by $\Sigma_k = \{0,1,\dots,k-1\}$ the set of digits base $k$, and by $\Sigma_k^* = \bigcup_{l \geq 0} \Sigma_k^l$ the set of words over $\Sigma_k$, including the empty word $\epsilon$. With the operation of concatenation, $\Sigma_k^*$ is a monoid. If $w = (w_i)_{i=0}^{l-1} \in \Sigma_k^*$, then by $[w]_k \in \NN_0$ we denote the corresponding integer $\sum_{i=0}^{l-1} w_i k^i$, and for $n \in \NN_0$ by $(n)_k \in \Sigma_k^*$ we denote the expansion of $n$ base $k$ with no leading $0$'s. (In particular, $(0)_k = \epsilon$.) Similarly, for $n \in \NN_0$ and $t \in \NN_0$ by $(n)_k^t \in \Sigma_k^t$ we denote the terminal $t$ digits of $n$ (padded with leading $0$'s if necessary).

A finite $k$-automaton with output $\cA$ (which we will subsequently just call automaton) consists of the following data:
\begin{enumerate}
\item a finite set of ``states'' $S$,
\item a distinguished ``initial'' state $s_0 \in S$,
\item a ``transition'' function $\delta \colon S \times \Sigma_k \to S$,
\item an ``output'' function $\tau \colon S \to \Delta$ (where $\Delta$ is some finite set).
\end{enumerate}

For instance, the Thue--Morse sequence, given by $t(n) = s_2(n) \bmod 2$ where $s_2(n)$ denotes the sum of binary digits of $n$, can be computed by the following automaton with $S = \{s_0,s_1\}$, $\delta(s_i,0) = s_i$, $\delta(s_i,1) = s_{1-i}$ and $\tau(s_i) = i$ for $i \in \{0,1\}$.
\begin{center}
\begin{tikzpicture}[shorten >=1pt,node distance=2cm, on grid, auto] 
   \node[state] (s_0)   {$s_{0}$}; 
   \node[state] (s_1) [right=of s_0] {$s_1$}; 
  \tikzstyle{loop}=[min distance=6mm,in=210,out=150,looseness=7]
  
    \path[->] 
    
    (s_0) edge [loop left] node {0} (s_0)
          edge [bend right] node [below]  {1} (s_1);
          
 \tikzstyle{loop}=[min distance=6mm,in=30,out=-30,looseness=7]
 \path[->]
    (s_1) edge [bend right] node [above]  {1} (s_0)
          edge [loop right] node  {0} (s_1);
\end{tikzpicture}
\end{center}

We will also occasionally be interested in  automata without output, or without a distinguished initial state; in this case we will refer to them as partial automata (it will always be clear from the context which data is present).

Given an automaton $\cA = (S,s_0,\delta,\tau)$ we may extend $\delta$ to $\Sigma_k^*$ by requiring that $\delta(s, uv) = \delta( \delta(s,v), u)$ for all $s \in S,\ u,v \in \Sigma_k^*$. The automaton induces a function $a \colon \Sigma_k^* \to \Delta$ given by $a(w) = \tau( \delta(s_0, w))$, and hence also a sequence (which we will denote by the same letter) $a\colon \NN_0 \to \Delta$, $a(n) = a((n)_k) = \tau( \delta(s_0, (n)_k))$. Sequences (resp. functions) $a(\cdot)$ which arise this way are said to be $k$-\emph{automatic}.

The class of $k$-automatic sequences is closed under arithmetic operations and restriction to arithmetic progressions for any $k \geq 2$. That is, if $a(n)$ and $b(n)$ are $k$-automatic sequences taking values in $\CC$,  then $a(n)+b(n),\ a(n) \cdot b(n)$ are $k$-automatic; and if $a(n)$ is any $k$-automatic sequence and $q \in \NN,\ r \in \NN_0$ then $a(qn+r)$ is $k$-automatic. 

A (partial) $k$-automaton $\cA = (S,\delta)$ without output and initial state is \emph{strongly connected} is there exists a path between any pair of vertices, i.e.\ for each $s, s' \in S$ there exists $v \in \Sigma_k^*$ with $\delta(s,v) = s'$. A strongly connected component of $\cA$ is a set $S' \subset S$ of states such that for any states $s,s' \in S$ there exists $v \in \Sigma_k^*$ such that $\delta(s,v) = s'$, i.e.\ $S'$ is strongly connected as a directed multigraph.

We will usually treat the base $k \geq 2$ as fixed, although occasionally it will be convenient to replace it by a power $k^t$. If a $k$-automatic sequence $a(n)$ is given, then this is essentially the only freedom we have in the choice of $k$.  Indeed, we have the following result.

\begin{theorem}[Cobham]
	Let $k,l \in \NN_{\geq 2}$, and let $a(n)$ be a $k$-automatic sequence. Then $a(n)$ is $l$-automatic if and only if either $a(n)$ is ultimately periodic or $\log l/ \log k \in \QQ$.
\end{theorem}

Here, a sequence is \emph{ultimately periodic} if it agrees with a periodic sequence away from a finite set.

A slight technical difficulty stems from the fact that elements of $\Sigma_k^*$ may well have leading $0$'s. Luckily, whenever an automatic sequence $a \colon \NN_0 \to \Delta$ is given, it is always possible to find an automaton $\cA$ which produces $a(n)$ so that the corresponding sequence $a\colon \Sigma_k^* \to \Delta$ 
has the property that $a(w) = a([w]_k)$ for all $w \in \Sigma_k^*$, i.e.\ $\tau(\delta(s,0)) = \tau(s)$ for all $s$. In this case, we will say that $\cA$ \emph{ignores leading $0$'s}. We will assume that all our automatic sequences $a \colon \NN_0 \to \Delta$ are produced by automata which ignore leading $0$'s.

For any sequence $a \colon \Sigma_k^* \to \Delta$, we define the $k$-\emph{kernel} of $a$ to be the set $\cN_k(a)$ of sequences of the form $b(u) = a(uv)$ where $v \in \Sigma^*_k$. Accordingly, for any sequence $a \colon \NN_0 \to \Delta$, we define the $k$-kernel $\cN_k(a)$ of $a$ to be the set of sequences $b(n) = a(k^l n + m)$ where $m < k^l$. Note that these definitions are consistent with the way that we identify sequences $\Sigma_k^* \to \Delta$ and $\NN_0 \to \Delta$.

The relevance of $k$-kernels to the study of automatic sequences stems from the following well-known characterisation.

\begin{proposition}
	Let $a \colon \Sigma_k^* \to \Delta$. Then $a$ is $k$-automatic if and only if $\cN_k(a)$ is finite. The analogous statement holds for sequences $\NN_0 \to \Delta$.
\end{proposition}

We will also use the complementary notion of ``co-kernel''. For $a \colon \Sigma_k^* \to \Delta$, the co-kernel $\cN'_k(a)$ consists of the sequences $b(u) = a(vu)$ where $v \in \Sigma^*_k$. This notion does not have a satisfactory analogue for sequences $\NN_0 \to \Delta$.

Let $R \colon \Sigma_k^* \to \Sigma_k^*$ denote the reflection, i.e.\ $R(w_{l-1}\dots w_{1} w_0) = w_0w_1 \dots w_{l-1}$. It is well known that $a \colon \Sigma_k^* \to \Delta$ is $k$-automatic if and only if $a \circ R$ is automatic. Since $\cN'_k(a)$ consists of sequences of the form $b \circ R$ where $b \in \cN_k(a \circ R)$, we conclude that $a \colon \Sigma_k^* \to \Delta$ is automatic if and only if $\cN'_k(a)$ is finite.

Suppose that the $k$-automatic sequence $a\colon \Sigma_k^* \to \Omega$ is produced by the automaton $\cA = (S,s_0,\delta,\tau)$. Then, the sequences in $\cN_k(a)$ are obtained by changing the initial state, and the sequences in $\cN_k'(a)$ are obtained by changing the output function. More precisely, if $b \in \cN_k(a)$ and $b' \in \cN_k'(a)$ are given by $b(u) = a(uv)$ and $b'(u) = a(vu)$, then $b$ is produced by the automaton $(S, \delta(s_0,v), \delta, \tau)$, and $b'$ is produced by $(S,s_0,\delta,\tau')$ where $\tau'(s) = \tau( \delta(s,v))$. 

In particular, given a ``partial automaton'' consisting of a set of states $S$ and a transition function $\delta$, as well as a (not necessarily finite) target set $\Omega$, if we let $\cM$ denote the family of sequences $a \colon \Sigma_k^* \to \Omega$ produced by all possible automata $(S,s_0,\delta,\tau)$ where $s_0 \in S$ and $\tau \colon S \to \Omega$, then $\cM$ is closed under the operation of taking kernels and co-kernels.

{
\subsection*{Measure-preserving systems}
By a measure-preserving system we mean a triple $(X,\mu,T)$, where $(X,\mu)$ is a probability space and $T:X\to X$ is a $\mu$-preserving transformation. For every $p\geq 1$ one calls the corresponding map $T\colon L^p(\mu)\to L^p(\mu)$ defined by $(Tf)(x):=f(Tx)$ the \emph{Koopman operator}; the Koopman operator is a linear isometry.

A measure-preserving system $(X,\mu,T)$ is called \emph{ergodic} if for measurable sets $T^{-1}(A)\subset A$ implies $\mu(A)\in\{0,1\}$. For the Koopman operator $T$ this means that the space of $T$-invariant functions $\Fix(T)$ consists of constant functions only, i.e., $\dim \Fix( T) =1$. Moreover, $(X,\mu,T)$ is called \emph{totally ergodic} if $T^n$ is ergodic for every $n\in\NN$. The equivalent spectral condition is that the Koopman operator $T$ does not have any rational eigenvalue on the unit circle other than $1$ and $\dim \Fix( T)=1$. 

For the basic theory of measure-preserving transformations we refer to any book on ergodic theory, e.g., to Walters \cite{W82}, Petersen \cite{P89} or \cite{EFHN15}.
}

\makeatletter{}\section{Outline}\label{sec:outline}

In this section, we outline the main argument, and explain how the proofs of Theorems \ref{thm:A}, \ref{thm:B} and \ref{thm:D} can be reduced to Fourier analysis.  We also discuss some examples, showing where our methods are (or are not) applicable.

\renewcommand{\SS}{\mathbb{S}}
 When it comes to $L^2$-convergence, the reduction is rather straightforward. Indeed, if $(X,\mu,T)$ is a measure-preserving system and $f \in L^2(\mu)$, then by the Spectral Theorem the space spanned by $T^n f$ for $n \in \NN_0$ can be identified with a subspace of $L^2(\SS,\nu_f)$ for some measure $\nu_f$ on the complex unit circle $\SS = \set{z \in \CC}{\abs{z}=1}$ through a map induced by $T^n f \mapsto z^n$.

\begin{proposition}\label{prop:reduction:B}
Let $p\colon \NN_0 \to \NN_0$ and $a \colon \NN_0 \to \CC$ be any sequences.
Suppose that $a(n)$ is bounded and that for each $\alpha \in \RR \setminus \QQ$ we have
	\begin{equation}
		\lim_{N \to \infty} \EE_{n < N} a(n) e(\alpha{  p(n)}) = 0.
		\label{eq:280}
	\end{equation}
	Then for any totally ergodic measure-preserving system $(X,\mu, T)$, and $f \in L^2(\mu)$ { with $\int_X f d\mu = 0$} we have
	\begin{equation}
		\EE_{n < N} a(n) T^{{ p(n)}} f \to 0 \text{ in } L^2.
		\label{eq:281}
	\end{equation}
\end{proposition}
\begin{proof}
	This is a standard application of the Spectral Theorem.
\end{proof}

 Hence, Theorem \ref{thm:A} will follow as soon as we can prove that \eqref{eq:280} holds for any balanced automatic sequence $a \colon \NN_0 \to \CC$ and polynomial sequence $p \colon \NN_0 \to \NN_0$. This is carried out in Section \ref{sec:Mean-Convergence}.

For pointwise convergence, more precise estimates are needed. Note also that for this part we restrict our attention to the case when $p$ is a polynomial.

\begin{proposition}\label{prop:reduction:A}
	{ Let $p \in \ZZ[x]$ be a polynomial with $p(\NN_0) \subset \NN_0$.} Suppose that $a(n)$ is a bounded sequence in $\CC$ with the property that 
	\begin{equation}
		\sup_{ \alpha \in \RR} \abs{ \EE_{n < N} a(n) e(\alpha { p(n)})} = O( 1/\log^2 N).\label{eq:860}		
	\end{equation}
	Then, for every measure-preserving system $(X,\mu, T)$ and every $f \in L^r(\mu)$, $r>1$, there exists a set $X' \subset X$ with $\mu(X') = 1$, such that
	\begin{equation}
		\EE_{n < N} a(n) T^{{ p(n)}} f(x) \to 0
    \label{eq:conv-pol}
	\end{equation}
	for any $x \in X'$. If $p$ is linear, then (\ref{eq:conv-pol}) holds for every $f \in L^1(\mu)$.
\end{proposition}
\begin{proof}
For a proof, see Corollary 2 in \cite{Fan-2017+}. This also follows by {an adaptation of }
the proof of Proposition 3.1 in \cite{AHKLR}.
\end{proof}	
\noindent Note that the restriction $r>1$ in Proposition \ref{prop:reduction:A} is due to the fact that the pointwise ergodic theorem (and hence also the maximal inequality which is crucial for the proof) for nonlinear polynomials fails in $L^1$ in general, see Buczolich, Mauldin \cite{BuczolichMauldin} and La Victoire \cite{LaVictoire}.

\smallskip

As before, it follows that in order to prove Theorems \ref{thm:B} and \ref{thm:D}, it will suffice to verify that condition \eqref{eq:860} holds for the sequence $a$ and the appropriate class of polynomial sequences. This is carried out in Sections \ref{sec:Pointwise-convergence} and \ref{sec:Invertible}, respectively.

\mbox{}

As an example, we consider the Thue--Morse sequence, given by $t(n) = (-1)^{s_2(n)}$, where $s_2(n)$ denotes the sum of digits of $n$ base $2$. We use $\log$ to denote logarithm base $2$. The following lemma with a superior value of $c = 1 - \log 3/ \log 4$ can be found in \cite{Gelfond-1968}, but we present the following argument as a source of motivation.

\begin{lemma}\label{lem:cancellation-ThueMorse}
	Let $t(n)$ be the Thue--Morse sequence. There exists a constant $c > 0$ such that 
	\begin{equation}
		\sup_{\alpha \in \RR} \abs{ \EE_{n < N} t(n) e(n\a)} \ll N^{-c}.
	\end{equation}
\end{lemma}
\begin{proof}
	For $L \in \NN$ denote $A(L,\alpha) = \abs{ \EE_{n < 2^L} t(n) e(n\a) }$, and note that the identity $A(L, \alpha) = A(L-l, 2^l \alpha) A(l, \alpha)$ follows immediately from the relation $t( 2^l n + m) = t(n) t(m)$ for $n,m,l$ with $0 \leq n$, $0 \leq m < 2^l$.
	An elementary computation shows that 
	
	$$ A(2,\alpha) = \abs{\frac{(1-e(\a))(1-e(2\a))}{4}} = \abs{ \sin\bra{ \pi \alpha} \sin\bra{ 2 \pi \alpha}} \leq \frac{4}{3 \sqrt{3}}. $$
	 It follows by induction that $A(L,\a)  \leq \frac{4}{3 \sqrt{3}} A(L-2,\alpha) \ll 2^{ - c L}$ where $c = \frac{1}{2} \log \frac{3 \sqrt{3}}{4}$. 
	 
	 Now, for arbitrary $N$, note that $[N]$ can be decomposed into a disjoint union of intervals $I_j$ taking the form $I_j = [m_j 2^{l_j}, (m_j+1)2^{l_j})$ where $m_i,l_i \in \NN_0$ and $l_1 > l_2 > l_3 > \dots$. Such decomposition can be constructed greedily, taking in each step the largest possible value of $l_i$. 
Then, 

	 \begin{align*}
	 \abs{ \EE_{n < N} t(n) e(n\a ) }
	  &\leq \sum_j \frac{2^{l_j}}{N} \abs{ \EE_{n \in I_j} t(n) e(n \a) }
	  \\& = \sum_j \frac{2^{l_j}}{N} A(l_j, \a) 
	  \ll  \sum_j \frac{2^{l_j(1-c)}}{N} \ll N^{-c}. \qedhere
	  \end{align*}
\end{proof}

As a consequence, the conclusion of Theorem \ref{thm:D} (hence also \ref{thm:A} and \ref{thm:B}) holds for the Thue--Morse sequence. (The Thue--Morse sequence is invertible and totally balanced, see Sec.\ \ref{sec:Invertible}.)

\mbox{}

It is natural to ask about the degree to which our results can be extended. Theorem \ref{thm:A} deals with arbitrary automatic sequences. The assumption that $\int_X f d\mu = 0$ cannot be relaxed, since automatic sequences need not be Ces\`aro convergent. (Note, however, that for Ces\`aro convergent sequences this condition is irrelevant since we can replace $f$ with $f - \int_X f d\mu$.) For similar reasons, the assumption of total ergodicity cannot be dropped (pick $f$ with $T^q f = f$ and pick $p$ such that $q| p(n)$ for all $n$).

Theorem \ref{thm:B} cannot be extended to all automatic sequences for similar reasons, as shown in the following proposition. Here, $\la f,g \ra$ is shorthand for $\int_X f(x) \bar g(x) d\mu(x)$.

\begin{proposition}
Let $a \colon \NN_0 \to \{0,1\}$ be the sequence given by $a(0) = 0$ and 

$$a(n) = \floor{ \log_2 n} \bmod 2$$
for $n \geq 1$, so that $a(n) = 1$ if the length of $n$ base $2$ is odd, and $a(n) = 0$ otherwise.  Then $a$ is $2$-automatic and Ces\`aro divergent. Moreover, for every measure-preserving system $(X,\mu,T)$ and every $f\in L^1(\mu)$, the following assertions are equivalent:
\begin{itemize}
\item[(i)] $\la f,g \ra=0$ for every $g\in \Fix(T)\cap L^\infty(\mu)$;
\item[(ii)] the weighted ergodic averages 
\begin{equation}\label{eq:weighted-ave}
\EE_{n < N} a(n) f(T^nx)
\end{equation}
converge to $0$ almost everywhere.
\end{itemize}
\end{proposition}
\begin{proof}
It is straightforward to construct a $2$-automaton which produces $a(n)$; it is enough to use two states as outlined below (the initial state is $s_0$ and the output at $s_i$ is $i$). 
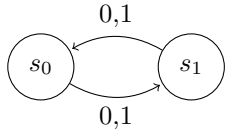
\begin{center}
\begin{tikzpicture}[shorten >=1pt,node distance=2cm, on grid, auto] 
   \node[state] (s_0)   {$s_{0}$}; 
   \node[state] (s_1) [right=of s_0] {$s_1$}; 
  \tikzstyle{loop}=[min distance=6mm,in=210,out=150,looseness=7]
  
    \path[->] 
    
    (s_0) edge [bend right] node [below]  {0,1} (s_1);
          
 \tikzstyle{loop}=[min distance=6mm,in=30,out=-30,looseness=7]
 \path[->]
    (s_1) edge [bend right] node [above]  {0,1} (s_0);
\end{tikzpicture}
\end{center}

Let $A(N) := \EE_{n < N} a(n)$. Since $A\bra{2^{2l-1}}>1/2$ and $A\bra{2^{2l}}= A\bra{2^{2l-1}}/2$ for every $l\in\NN$, $(a(n))$ is Ces\`aro divergent. 

Let $(X,\mu,T)$ be a system and let $f\in L^1(\mu)$ satisfy (i). It follows from the von Neumann decomposition that we may decompose

$$
L^1(\mu)=\Fix(T) \oplus V $$
where $V$ is the closed span of $\set{ h - Th }{ h \in L^\infty(\mu)}$ in $L^1(\mu)$, and $\oplus$ denotes orthogonal sum. Hence, (i) implies that $f \in V$.
By

$$ 
\left|\EE_{n < N} a(n) T^nf \right|\leq \EE_{n < N} T^n|f|
$$
and the classical maximal inequality, the set of functions for which the averages (\ref{eq:weighted-ave}) converge a.e.~is closed in $L^1(\mu)$. So we can assume that $f=h-Th$ for some $h\in L^\infty(\mu)$. Moreover, we can assume without loss of generality that $\|h\|_\infty\leq 1$. Since $a(n)$ is constant on the interval blocks between any two consecutive powers of $2$ we obtain
\begin{eqnarray*}
\left\|\EE_{n < N} a(n) T^nf \right\|_\infty&=&\left\|\EE_{n < N} a(n) (T^nh-T^{n+1}h) \right\|_\infty\\
&\leq& \EE_{n < N} |a(n)-a(n-1)|+\frac2N \\
&\leq& \frac{2\log_2 N +2}N\to 0 \quad \text{as } N\to\infty
\end{eqnarray*}
proving (ii).

Conversely, assume that (ii) holds and that $\la f,g \ra\neq 0$ for some $g\in \Fix(T)\cap L^\infty(\mu)$. Then the averages

$$
\left\la \EE_{n < N} a(n) T^nf,g \right\ra =\EE_{n < N} a(n) \la T^nf,T^ng\ra = \EE_{n < N} a(n) \la f,g\ra
$$
diverge contradicting (ii).
\end{proof}

In Theorem \ref{thm:D} we impose a relatively strong condition of invertibility. To the best of our knowledge, the analogous result might hold for a wider class of sequences; in particular it is possible that the same statement holds for totally balanced sequences. However, our proof does not extend to such sequences because of the use of van der Corput lemma. See Section \ref{sec:Invertible} for further discussion.

\makeatletter{}\section{Preliminaries}\label{sec:Prelims}

In this section, we discuss some basic lemmas which will be useful in subsequent sections. The main new insight is that if $\alpha \in \RR$ and $\beta_n \in \RR$ takes finitely many values then the average $\EE_{n < N} e(n\alpha + \beta_n)$ cannot be close to $1$ for large $N$, unless $\alpha$ is very well approximable by rationals.

\begin{definition}[Equidistribution]
	A sequence $(x_n)_{0\leq n<N}$ taking values in a metric space $(X,d)$ equipped with a Borel probability measure $\mu$ is \emph{$\delta$-equidistributed} if for any Lipschitz continuous function $f \colon X \to \RR$ it holds that 
	\begin{equation}
		\abs{ \EE_{ n < N} f(x_n) - \int_X f d\mu } < \delta \norm{f}_{\mathrm{Lip}}, 	
	\end{equation}	 
	where $\norm{f}_{\mathrm{Lip}} = \norm{f}_\infty + \sup_{x \neq y} \frac{\abs{ f(x) - f(y) }}{ d(x,y) } $.
\end{definition}

The following fact is well-known, for instance it is a very special case of the Quantitative Kronecker Theorem in \cite{GreenTao-2012}.
\begin{proposition}\label{prop:dicho-linear}
	There exists a constant $C$ such that for any $\a \in \RR$, $N \in \NN$ and $\delta$ with $0 < \delta < 1/2$, one of the following holds:
	\begin{enumerate}
		\item the sequence $(n \alpha \bmod{1})_{ n < N}$ is $\delta$-equidistributed in $\RR/\ZZ$; or
		\item there exists $\frac{p}{q} \in \mathbb{Q}$ with $0 < q < 1/\delta^C$ such that $\abs{ \alpha - \frac{p}{q} } < 1/(\delta^C N)$.
	\end{enumerate}
\end{proposition}

The following elementary fact will be useful on several occasions.
\begin{lemma}
	For any $r \in \NN$, any $N \in \NN$, and any ${1}/{(100 r^2)}$-equidistributed sequence $(x_n)_{ n < N}$ taking values in  $\RR/\ZZ$, and any partition $[N] = S_1 \cup S_2 \cup \dots \cup S_r$, we have
	\begin{equation}\label{eq:bound-001}
		\sum_{i=1}^r \frac{\abs{S_i}}{N} \abs{ \EE_{n \in S_i} e(x_n) } \leq 1 - \frac{1}{6r^2}.
	\end{equation}
\end{lemma}
\begin{proof}
	Put $\b_i = \arg\bra{ \EE_{n \in S_i} e(n \a)}$ and $f(x) = \max_{1 \leq i \leq r} \cos(2 \pi x - \b_i)$, so that the sum in \eqref{eq:bound-001} becomes
	\begin{equation}\label{eq:bound-002}
		\sum_{i=1}^r \frac{\abs{S_i}}{N} \EE_{n \in S_i} \cos( 2\pi x_n - \b_i ) 
		\leq \EE_{n < N } f(x_n) 
		\leq \int_0^1 f(x) dx + \frac{\norm{f}_{\mathrm{Lip}}}{100r^2}.
	\end{equation}
	It is clear that $\norm{f}_{\mathrm{Lip}} \leq 1+2\pi$ and that $f(x) < \cos(\frac{2\pi}{4r}) < 1 - \frac{1}{2r^2}$ on a union of intervals of length $\geq \frac{1}{2}$, whence the right hand side of \eqref{eq:bound-002} is at most $ 1 - \frac{1}{4r^2} + \frac{1+2\pi}{100 r^2} < 1 - \frac{1}{6r^2} $.
\end{proof}

\begin{corollary}\label{cor:linear-dichotomy}
	For any $r \in \NN$, there exists a constant $Q$ such that for any $\a \in \RR$ and $N \in \NN$, one of the following holds:
	\begin{enumerate}
		\item\label{cond:ld-1} for any partition $[N] = S_1 \cup S_2 \cup \dots \cup S_r$ it holds that $		\sum_{i=1}^r \frac{\abs{S_i}}{N} \abs{ \EE_{n \in S_i} e(n \alpha) } \leq 1 - \frac{1}{6r^2}$;
		\item\label{cond:ld-2} there exists $\frac{p}{q} \in \mathbb{Q}$ with $0 < q \leq Q$ such that $\abs{ \alpha - \frac{p}{q} } < Q/N$. 
	\end{enumerate}	  
\end{corollary}

Another result which we will extensively use is the classical van der Corput inequality. One of its many formulation is the following (see e.g.~\cite[Lemma 1.4.3]{Tao-whatsnew2}).

\begin{lemma}[van der Corput inequality]\label{lem:vdCorput}
	Let $x \colon \NN_0 \to \CC$ be a sequence with $\abs{x(n)} \leq 1$ for all $n$. Then for any $H,N \in \NN$ with $H < N$ we have
\begin{equation}
\abs{ \EE_{n<N} x(n) }^2 \ll
{ \EE_{h<H} \abs{ \EE_{n<N} x(n+h) \bar{x}(n) } } + O( H/N ).
\end{equation}	
All implicit constants are absolute (i.e.\ do not depend on $x$, $N$ and $H$).
\end{lemma}

As an immediate consequence, we note that (with the above notation) to prove that $\EE_{n<N} x(n) \to 0$ as $N \to \infty$, it suffices to prove that for all $h$ except for a set of density $0$ we have $\EE_{n<N} x(n+h) \bar x(n) \to 0$ as $N \to \infty$.

\makeatletter{}\section{Mean convergence}\label{sec:Mean-Convergence}

In this section we finish the proof of Theorem \ref{thm:A}. 
The main technical tool used for this purpose is the following proposition. We could derive this result directly from earlier work by Mauduit \cite[Thm.\ 1]{Mauduit-1986}, but we present an independent argument which motivates the approach we take in subsequent sections when proving Theorems \ref{thm:B} and \ref{thm:D}. For yet another approach, see Remark \ref{rmrk:alter-proof}.

\begin{proposition}\label{prop:cancellation-qualitative}
	Let $a\colon \NN_0 \to \CC$ be a $k$-automatic sequence, and let $\a \in \RR \setminus \QQ$. Then 
	\begin{equation}
		\label{eq:853}
		\lim_{N \to \infty} \EE_{n < N} a(n) e(n\a) = 0.
	\end{equation}
\end{proposition}
\begin{proof}
	For an automatic sequence $a \colon \Sigma^*_k \to \CC$, denote 
	
	\begin{align*}
		A(L,a,\a) &=  \abs{ \EE_{u \in \Sigma_k^L } a(u) e\bra{ [u]_k \a } },\\
		A(a,\a) &= \displaystyle \limsup_{L \to \infty} A(L,a,\a),\\
		A(a) &= \displaystyle \sup_{\a \in \RR \setminus \QQ } A(a,\a).
	\end{align*}

	 Our first goal is to prove that $A(a) = 0$ for any choice of $a$. Observe that
	 
	\begin{align*}
		A(L,a,\a)  & = \abs{ \EE_{ u \in \Sigma_k^{L-l}  } \EE_{ v \in \Sigma_k^l  } a\bra{ u v } e( k^l [u]_k \a) e( [v]_k \a) } 
		\\ & \leq \sum_{b \in \cN_k(a) } \frac{\abs{S_b^l}}{k^l} A(L-l,b,k^l\a) \abs{ \EE_{m \in S_b^l}  e(m \a)} ,
			\end{align*}
	where $S_b^l = \set{ [v]_k }{ v \in \Sigma_k^l ,\ a\bra{ u v} = b\bra{u} \text{ for all } u \in \Sigma_k^*}$. Fixing the value of $l$ and sending $L$ to infinity, we conclude that
	 
	\begin{align*}
	A(a,\a) & \leq \sum_{b \in \cN_k(a) }  \frac{\abs{S_b^l}}{k^l} \abs{ \EE_{m \in S_b^l}  e(m \a) } A(a,k^l\a).
	\end{align*}
	Using Corollary \ref{cor:linear-dichotomy} with $r = \abs{\cN_k(a)}$, and letting $l$ be sufficiently large that condition (\ref{cond:ld-2}) in Corollary \ref{cor:linear-dichotomy} does not hold, we obtain
	
	\begin{align*}
	A(a,\a) & \leq  (1- c) \max_{b \in \cN_k(a)} A(b),
	\end{align*}
	where  $c = {1}/\bra{6 \abs{\cN_k(a)^2}} > 0$. Letting $\a$ vary and repeating the same argument for all $b \in \cN_k(a)$, we conclude that 
	
	\begin{align*}
	 \max_{b \in \cN_k(a)} A(b) \leq (1-c) \max_{b \in \cN_k(a)} A(b),
	\end{align*}
	which is only possible when $A(a) = \max_{b \in \cN_k(a)} A(b) = 0$. In particular, we conclude that \eqref{eq:853} holds for $N$ restricted to powers of $k$.
	
	We now proceed to prove \eqref{eq:853} for arbitrary $N$. Assume that $a(u)$ ignores leading $0$'s, and identify it with a sequence $\NN_0 \to \CC$. 
	For any $N$, take $L = \floor{ \frac{9}{10} \log_k N}$ and $M = \floor{ N/k^L}$. Splitting $[N]$ into intervals of the form $[mk^L,(m+1)k^L)$ and $[Mk^L,N)$, we obtain
	
	\begin{align*}
	\EE_{n < N} a(n)e(n \alpha)  
	 & \leq \max_{b \in \cN'_k(a)} A(L,b,\alpha) + O\bra{1/M}
	\to 0 \text{ as } N \to \infty. \qedhere
	\end{align*}
	\end{proof}
	
\begin{corollary}\label{prop:cancellation-qualitative-polynomial}
Let $a \colon \NN_0 \to \CC$ be a $k$-automatic sequence, and let $p \in \RR[x]$ be a polynomial with at least one irrational coefficient other than the constant term. Then 
	\begin{equation}
		\lim_{N \to \infty} \EE_{n < N} a(n) e( p(n) ) = 0.
	\end{equation}
\end{corollary}
\begin{proof} 
	Splitting $[N]$ into a union of arithmetic progressions, we may assume that the leading coefficient of $p$ is irrational. We proceed by induction on $\deg p$. The case when $\deg p = 1$ follows easily from Proposition \ref{prop:cancellation-qualitative}.
	
	If $\deg p \geq 2$, then using the van der Corput Lemma, it will suffice to verify that for each $h \in \NN$,
	\begin{equation}
		\limsup_{N \to \infty} \abs { \EE_{ n < N} a(n) \bar{a}(n+h) e (\Delta_h p (n) ) } = 0,
	\end{equation}
	where $\Delta_h p(x) = p(x+h) - p(x)$. Since $a(n) \bar{a}(n+h)$ is a $k$-automatic sequence, $\deg \Delta_h p = \deg p - 1$ and the leading coefficient of $\deg \Delta_h p$ is irrational, this follows from the inductive claim.
\end{proof}

\begin{remark}\label{rmrk:alter-proof}
As was pointed out by one of the referees, an alternative strategy to prove Proposition \ref{prop:cancellation-qualitative} would be to use the fact that
automatic sequences come from fixed points of constant length substitutions, see  \cite[Cobham's Thm.\ 5.1]{Queffelec-2010}. Since for primitive substitutions the corresponding system is uniquely ergodic \cite[Michel's Thm.\ 5.6]{Queffelec-2010} and has only rational measurable ($=$ continuous) eigenvalues \cite[Prop.\ 6.1 and Thm.\ 6.2]{Queffelec-2010}, it remains to reduce to primitive systems, cf.~\cite[Robinson's Thm.\ 4.10]{Queffelec-2010}. Moreover, using Furstenberg's product construction \cite[Sect.\ 4.4.3]{EinsiedlerWard} one can deduce Corollary \ref{prop:cancellation-qualitative-polynomial} without using the van der Corput Lemma. We do not go into the details here. 
\end{remark}

\begin{proof}[Proof of Theorem \ref{thm:A}]
	Immediate from Corollary \ref{prop:cancellation-qualitative-polynomial} and Propositon \ref{prop:reduction:B}.
\end{proof} 

\makeatletter{}\section{Growth rate of partial sums}

In this section we prove a lemma describing possible growth of partial sums of automatic sequences. Questions of this type have been extensively studied and our estimate is rather standard, but we provide a detailed proof for the convenience of the reader.

To provide context, let us introduce the notion of a $k$-regular sequence (first put forward in \cite{AlloucheShallit-1992}). A sequence $a \colon \NN_0 \to \ZZ$ is said to be $k$-\emph{regular} if $\cN_k(a)$ is a finitely generated $\ZZ$-module; the same definition makes sense with other domains in place of $\ZZ$. Any $k$-automatic sequence is automatically $k$-regular, and it can be shown that conversely a finitely valued $k$-regular sequence is $k$-automatic \cite[Theorem 16.1.5]{AlloucheShallit-book}. Moreover, if $a(n)$ is a $k$-regular sequence, then the sequence of partial sums $(\Sigma a)(n) = \sum_{m < n} a(m)$ is again $k$-regular. In particular, partial sums of automatic sequences are regular.

In \cite{BellCoonsHare-2014}, Bell, Coons and Hare showed that if $a(n)$ is an unbounded regular sequence then $\abs{a(n)} \gg \log n$ for infinitely many $n$. A more precise description was obtained by the same authors in \cite{BellCoonsHare-2016}: $\limsup_{n \to \infty} \log \abs{a(n)}/\log \log n \in \NN \cup \{\infty\}$. Related results are also obtained in \cite{Dumas-2013,Dumas-2014}.

\begin{proposition}\label{prop:violet=>very-violet}
	Fix $k \geq 2$. Let $a \colon \NN_0 \to \CC$ be a $k$-automatic sequence with $\EE_{n < N} a(n) = o(N)$. Then, there exists a constant $c > 0$ such that $\EE_{n < N} a(n) = O(N^{-c})$.
\end{proposition}
\begin{proof}
	Fix the choice of $a$, and assume it is generated by an automaton which ignores leading $0$'s. Our first goal is to show that there exists $c > 0$ such that for all $b \in \cN'_k(a)$ we have the bound 
\begin{equation}
	\EE_{u \in \Sigma_k^L} b(u) = O( k^{-cL}).
	\label{eq:bound-435}
\end{equation}

		Because $\cN'_k(a)$ is finite, it suffices to prove the claim for a single $b \in \cN'_k(a)$. Let $b(u) = a(vu)$ for some $v \in \Sigma_k^*$. Writing $m = [v]_k$ we have
		
	\begin{align*}
		\EE_{u \in \Sigma_k^L} b(u) 
		&= \EE_{n < k^L} a(k^L m + n) 
		= \frac{1}{k^L} \sum_{n < (m+1) k^L} a(n) - \frac{1}{k^L} \sum_{n < m k^L} a(n) = o(1),
	\end{align*}
	as $L \to \infty$. Denote $s_b(L) := \EE_{u \in \Sigma_k^L} b(u)$. Using automaticity of $b(u)$, we obtain a linear recurrence $s_b(L+1) = \sum_{d \in \cN'_k(a)} \alpha_{b,d} s_d(L)$ for some coefficients $\alpha_{b,d} \in \RR_{\geq 0}$. Since a recursive sequence tending to $0$ tends to $0$ at an exponential rate, we obtain \eqref{eq:bound-435}.
	
	We have thus proved the desired bound for $N = k^L$. For general $N$, pick $L = \floor{ \frac{9}{10} \log N}$, $M = \floor{N/k^L}$ so that we may bound
	
\begin{align*}
	\EE_{n < N} a(n) 
	&= \EE_{m < M} \EE_{n < k^L} a(k^L m + n) + O\bra{N^{-1/10}}
	\\&= O\bra{ \max_{b \in \cN'_k(a)} \abs{ \EE_{u \in \Sigma_k^L } b(u) }} + O\bra{N^{-1/10}}
	= O\bra{ N^{-9c/10} + N^{-1/10} }. \qedhere
\end{align*} 
\end{proof}

Recall that we call a sequence $a \colon \NN_0 \to \CC$ totally balanced if for any periodic sequence $b \colon \NN_0 \to \CC$ we have $\EE_{n < N} a(n) b(n) = o(1)$ as $N \to \infty$ (where the rate of convergence is of course allowed to depend on $b(n)$). 

\begin{corollary}\label{cor:antiper=>strongly-antiper}
	Suppose that $a(n)$ is a totally balanced automatic sequence. Then for any periodic sequence $b(n)$ there exists a constant $c > 0$ such that 
	
	$$\EE_{n < N} a(n) b(n) = O(N^{-c}).$$
\end{corollary}
\begin{proof}
	Immediate application of Proposition \ref{prop:violet=>very-violet} to $a(n)b(n)$.
\end{proof}

\makeatletter{}\section{Aperiodicity}

Even though a totally balanced sequence $a(n)$ is guaranteed to have mean  $0$ along any arithmetic subsequence in the sense that 

$$\EE_{n<N} a(qn+r) = o(1),$$ 
it is by no means guaranteed that various sequences $a(qn + r)$ ($0 \leq r < q$) obtained by restricting $a(n)$ to arithmetic subsequences are in any way related. 

For instance, the sequence $a(n) = t(n) (1-(-1)^n)/2$ is clearly $2$-automatic, and it is not hard to verify that it is totally balanced. It is also easy to notice that $a(2n) = t(n)$ while $a(2n+1) = 0$. A much closer relation exists between the sequences $t(qn+r)$ for fixed $q,r$ with $0 \leq r < q$. Hence, in proving Theorem \ref{thm:B} for the sequence $a(n)$ it will be more convenient to work instead with $a(2n)$ and $a(2n+1)$ independently. The goal of this section is to obtain a similar decomposition for a general totally balanced automatic sequence.

For a $k$-automaton $\cA = (S,s_0,\delta,\tau)$ we may consider the frequencies 

$$ \pi(s,s') = \lim_{L \to \infty} \frac{ \abs{ \set{ u \in \Sigma_k^L}{ \delta(s, u ) = s'}} }{k^L}, $$
which may or may not exist. More generally, for $q \in \NN$, $r \in \NN_0$ with $0 \leq r < q$, let

$$
	\pi(s,s'; r(q)) = \lim_{L \to \infty} \frac{ \abs{ \set{ u \in \Sigma_k^L}{ \delta(s, u ) = s',\ [u]_k \equiv r \bmod{q} }} }{ \abs{ \set{ u \in \Sigma_k^L}{ [u]_k \equiv r \bmod{q} }} }.
$$
If $s = s_0$, we simply write $\pi(s')$ or $\pi(s',r(q))$. Note that these quantities depend only on $(S,s_0,\delta)$, which we will call a $k$-automaton without output.

For lack of a better phrase, we shall say that an automaton without output $\cA = (S,s_0,\delta)$ is \emph{strongly aperiodic} if for each $q \in \NN$, for each $r \in \NN_0$ with $r < q$ and $s \in S$, the frequencies $\pi(s, r(q))$ exist and are equal to $\pi(s)$. (To motivate this piece of nomenclature, note, in particular, that a non-constant automatic sequence produced by a strongly aperiodic automaton does not become periodic even after restricting to an arithmetic progression; hence strong aperiodicity can be seen as a far reaching strengthening of the property of not being periodic. This property can also be viewed as an analogue of the notion of aperiodicity for graphs.) We will always assume that all states $s \in S$ are reachable from the initial state $s_0$. Under this assumption, if $\cA$ is aperiodic then it is also strongly connected. Note that for a strongly connected automaton, changing the initial state does not alter strong aperiodicity; we will say that a strongly connected automaton $(S,\delta)$ without a distinguished initial state is strongly aperiodic if $(S,s_0,\delta)$ is aperiodic for some (all) $s_0 \in S$.

It is relevant to the study of aperiodic behaviour of an automaton $\cA$ to know what the possible values associated with the cycles are. Let $\cA = (S,\delta)$ be an automaton without output nor initial state. For $s \in S$, we will consider the sets

$$
	D_{\cA,s} := \set{ [u]_k - [v]_k }{ \delta(s,u) = \delta(s,v) = s,\ \abs{u} = \abs{v} },
$$
and put $d_{\cA,s} = \gcd(D_{\cA,s})$.

\begin{lemma} Fix $k \geq 2$.
	Let $\cA = (S,\delta)$ be a strongly connected $k$-automaton without output and initial state. Then $d_{\cA,s}$ does not depend on $s$ and is coprime to $k$.
\end{lemma}

We will denote the common value of $d_{\cA,s}$ by $d_\cA$.

\begin{proof}
	To verify that $k$ is coprime to $d_{s,\cA}$, take two $u,v \in \Sigma_k^*$ such that $u$ ends with $0$, $v$ ends with $1$, and $\delta(s,u) = \delta(s,v) = s$. Such $u,v$ exist because $\cA$ is strongly connected. Replacing $u,v$ with $u^{\abs{v}}$ and $v^{\abs{u}}$, we may assume that $\abs{u} = \abs{v}$. Hence, $[u]_k - [v]_k \in D_{\cA,s}$ and is coprime to $k$, and $d_{\cA,s}$ is coprime to $k$.
	
	To verify that $d_{\cA,s}$ is independent of $s$, let us fix first some $s,s' \in S$. Pick $x,y \in \Sigma_k^*$ such that $\delta(s,x) = s'$ and $\delta(s',y) = s$. Then, for any $u,v$ as in definition of $D_{\cA,s}$ we have that 
	
	$$[ xuy]_k - [xvy]_k = k^{\abs{y}}\bra{[ u]_k - [v]_k} \in D_{\cA,s'}.$$
Hence, there exists $m = m(s,s')$ such that $d_{\cA,s'} \mid k^m d_{\cA,s}$. Since $d_{\cA,s'}$ is coprime to $k$, $d_{\cA,s'} \mid d_{\cA,s}$. By symmetry, $d_{\cA,s'} = d_{\cA,s}$.
\end{proof}

\begin{proposition}\label{prop:very-primitive=>strongly-aperiodic} Fix $k \geq 2$. 
Let $\cA = (S,\delta)$ be a strongly connected automaton (without output and initial state). Suppose that $d_{\cA} = 1$ and that there exists a state $s \in S$ such that $\delta(s,0) = s$. Then $\cA$ is strongly aperiodic. 
\end{proposition}
\begin{proof}

	We wish to show that $\pi\bra{s,s'; r(q)} = \pi(s')$ for any $r,q$ with $0 \leq r < q$ and any $s,s' \in S$, and both quantities exist.

	Note that if the claim holds for a given value of $q$, then it also holds for any $q'$ which divides $q$. Hence, we may assume without loss of generality that $q = q_0 q_1$ where $q_0 = k^m$ and $q_1 = k^{l}-1$ for some $l,m \in \NN$. Note further that for any $r$ with $0 \leq r < q$ there exist some $w \in \Sigma_k^m$ and $r_1$ with $0 \leq r_1 < q_1$ such that for all $s,s' \in S$ we have
	
	$$
		\pi\bra{ s,s' ; r(q)} = \pi\bra{ \delta(s,w), s'; r_1(q_1)}
	$$
(in the sense that if one of the quantities is defined then so is the other, and if so they are equal). Hence, we may without loss of generality assume that $q_0 = 1$, whence $q = k^l -1$. Fix the value of $q$ from now on.
	
	We will consider the random walk $\mathcal{W}$ on the vertex set $S \times \ZZ/q\ZZ$, where at each step we select $u \in \Sigma_k^l$ randomly and pass from $(s,n)$ to $(\delta(s,u),n+[u]_k)$. Note that a sequence $u = u_{t-1} u_{t-2} \dots u_0 \in \bra{\Sigma_k^{l}}^t \simeq \Sigma_k^{tl}$ gives rise to a path from $(s,n)$ to $(s',n+r)$ if and only if $\delta(s,u) = s'$ and $[u]_{k} \equiv r \bmod{q}$. 
	
	In order to apply the Perron--Frobenius theorem to $\mathcal{W}$, we need to verify that it is aperiodic (in the sense that the greatest common divisor of all cycles is $1$) and strongly connected (in the sense that there exists a path from any vertex $(s,r)$ to any other vertex $(s',r')$). The former condition is clear because for any $s \in S$ such that $\delta(s,0) = s$, we have a loop in $\mathcal{W}$ at $(s,0)$ corresponding to taking $u = 0$. 
	
	We now proceed to prove strong connectedness. Note that a path from a vertex $(s,r)$ to a vertex $(s',r')$ exists if and only if there exists a path from the vertex $(s,0)$ to $(s',r'-r)$. Hence, for a fixed choice of $s \in S$, the set $I_s$ of $r \in \ZZ/q\ZZ$ such that there exists a path from $(s,0)$ to $(s,r)$ is a subgroup of $\ZZ/q\ZZ$.
	
	We will show that actually $I_s = \ZZ/q\ZZ$. Pick any $n \in D_{\cA,s}$ and let $u,v \in \Sigma_k^*$ be two words with $\delta(s,u) = \delta(s,v) = s$ and $\abs{u} = \abs{v}$ with $n = [u]_k - [v]_k$. Let $w_0 = u^{l}$ and $w_1 = u^{l-1}v$. By construction, there exists a path from $(s,0)$ to $(s,[w_0]_k)$ and $(s,[w_1]_k)$, whence $n = [u]_k - [v]_k = [w_0]_k - [w_1]_k \in I_s$. If follows that $1 = d_{\cA} \in I_s$ and $I_s = \ZZ/q\ZZ$, as claimed.

	To prove that $\mathcal{W}$ is strongly connected, it will now suffice to verify that for each $s,s' \in S$, there exists a path from $(s,0)$ to $(s,r')$ for some choice of $r'$. This is equivalent to the statement that for any $s,s' \in S$, there exists $u \in \Sigma_k^*$ with $l \mid \abs{u}$ and $\delta(s,u) = s'$. If either $\delta(s,0) = s$ or $\delta(s',0) = s'$, then it is enough to pick any $v \in \Sigma_k^*$ with $\delta(s,v) = s'$ and set $u =  \abs{v}0^{(l-1)\abs{v}}$ (or $u = 0^{(l-1)\abs{v}} \abs{v}$). Otherwise, pick $s''$ with $\delta(s'',0) = s''$ and note that there exists a path from $(s,0)$ to $(s'', r'')$ and from $(s'',r'')$ to $(s',r')$ for some $r',r'' \in \ZZ/q\ZZ$.

	It now follows from Perron--Frobenius that there exist unique limiting probabilities for the random walk $\mathcal{W}$ and they are independent of the starting point. Hence, the limit
	
	\begin{align*}
		\pi'(s,s';r(q)) 
		&= \lim_{K \to \infty} \PP \bra{\mathcal{W} \text{ goes from $(s,q-r)$ to $(s',0)$ in $K$ steps} }
		\\ & = \lim_{K \to \infty} \frac{1}{k^{Kl}} \abs{ \set{u \in \Sigma_k^{Kl} }{ \delta(s,u) = s', [u]_k \equiv r \bmod{q} }}
	\end{align*}
	exists for any $r$ with $0\leq r < q$ and $s,s' \in S$, and does not depend on $r$ and $s$. Denote the common value of $\pi'(s,s',r(q))$ by $\pi'(s')$.
	
	For any $j$ with $0 \leq j < l$ we may similarly compute that
	
	\begin{align*}
	 &\phantom{=} \lim_{K \to \infty} \frac{1}{k^{Kl+j}} \abs{ \set{u \in \Sigma_k^{Kl+j} }{ \delta(s,u) = s', [u]_k \equiv r \bmod{q} }}
	 \\ &= \EE_{v \in \Sigma_k^j} \pi\bra{ \delta(s,v), s', k^{l-j}(r-[v]_k) (q)} = \pi'(s').
	\end{align*}
	
	If follows that for any $s,s'$ and $r$, the limit $\pi(s,s',r(q))$ exists and equals $\pi'(s')$. In particular, $\pi(s)$ exists and $\pi(s) = \pi'(s) = \pi(s,s',r(q))$, and thus $\cA$ is strongly aperiodic.
\end{proof}

\begin{proposition}\label{prop:aperiodic=>strongly-aperiodic}
	Fix $k \geq 2$. Let $a(n)$ be a $k$-automatic sequence which is produced by a $k$-automaton $\cA = (S,s_0,\delta,\tau)$ which is strongly connected and ignores leading $0$'s.
	
	Then there exist $k'$ which is a power of $k$ and $q \in \NN$ such that for any $r$ with $0 \leq r < q$, the sequence $a_r(n) = a(qn + r)$ is produced by some $k'$-automaton $\cA_r$ which ignores leading $0$'s and has the property that any of its terminal components is strongly aperiodic. Moreover, $k'$ and $q$ depend only on $(S,\delta)$.
\end{proposition}
\begin{proof}

	We begin with a reduction to a case when $\cA$ has some additional favourable properties.
	
	Let $k' = k^l$ be a power of $k$. Then $a(n)$ is a $k'$-automatic sequence, and it is produced by the automaton $\cA' = (S',s_0',\delta',\tau')$ constructed as follows. The set of states $S'$ will be a subset of $S$, defined thereafter, and $s_0' = s_0$, $\tau' = \tau|_{S}$. The transition function is defined for $u \in \Sigma_k^l$ by $\delta'(s,[u]_k) = \delta(s,u)$. Finally, $S'$ is the set of states $s \in S$ reachable by $\delta'$ from $s_0$, or equivalently the set of states which are reachable in $\cA$ from $s_0$ by a path of length divisible by $l$. Note that under these definitions, for any $n \in \NN$, $\delta'(s_0, (n)_{k'}) = \delta(s_0, 0^j (n)_k)$ (with natural identifications, where $j$ is chosen so that the length of $0^j (n)_k$ is divisible by $l$), whence the condition of ignoring the leading $0$'s ensures that $\cA$ and $\cA'$ produce the same sequence. Moreover, $\cA'$ ignores the leading $0$'s and is strongly connected. Thus, we may freely replace $k$ with $k' = k^l$ and $\cA$ with $\cA'$. We will perform this replacement several times; to avoid obfuscating the notation we will reuse the same symbols $k$ and $\cA$. 
	
	Note that the set of states $S$ may become smaller as we change the base $k$. Replacing $k$ by its power, we may assume that $S$ has stabilised, i.e.\ that subsequent replacements will not decrease the set of states further. We may also assume that the action of $0$ is idempotent, in the sense that $\delta(s,00) = \delta(s,0)$ for any $s$ (this property is preserved under the change of base). Let $s_1 = \delta(s_0,0) \in S$ (change of base will not affect $s_1$). Because any two paths from $s_1$ to $s_1$ can be padded by an arbitrary number of $0$'s, the set $D_{\cA,s_1}$ does not change when the base is changed, and any subsequent change of base does not change $d_{\cA}$. We put $q = d_{\cA}$. We may assume that $k$ is much larger than $q$, and in particular that $\delta(s,0j) = \delta(s,j)$ for any $j$ with $0 \leq j < q$ and any $s \in S$. 
	
	We next construct the automata $\cA_r = (R, s'_r, \delta', \tau')$ which produce the sequences $a_r(n) = a(qn+r)$. They are defined as follows:
	
	\begin{align*} 
	R &= S \times [q], 
	 &s_r' &= 
	 (s_0,r),
	\\ \delta'\bra{ (s,m), j} &= \bra{ \delta(s, qj+m \bmod k), \floor{\frac{qj+m}{k}}},
	& \tau'(s,m) &= \tau( \delta(s,m)).
	\end{align*}
	(In the last line, note that $[q]$ may be considered as a subset of $\Sigma_k$.)

	It is routine to check that $R$ is preserved under $\delta(\cdot, j)$ for $j \in \Sigma_k$, that $\cA_r$ ignores the leading $0$'s, and that $\delta'(s',00) = \delta'(s',0)$ for all $s' \in R$. Moreover, an inductive argument shows that for any $n \in \NN$, we have
	
	$$\delta'(s'_r, (n)_k) = \bra{ \delta\bra{s_0, (q n + r )_k^t }, \floor{\frac{qn+r}{k^t}} },\quad t = \floor{\log_k n}+1.$$ 
	(Recall that $(n)_k^t$ denotes the terminal $t$ digits of $(n)_k$, padded by $0$'s if necessary.) 
	
	Hence, we may derive
	
	$$\delta'\bra{s'_r, 0(n)_k} = \bra{\delta\bra{s_0,0(qn+r)_k},0}$$
	 and $\cA_r$ indeed produces the sequence $a_r$. Slightly more generally, we have
	 
	 $$\delta'\bra{(s,r), 0(n)_k} = \bra{\delta\bra{s,0(qn+r)_k},0}.$$
	Note that there is no guarantee that $\cA_r$ is strongly connected for any choice of $r$, or even that all states $s' \in R$ are reachable from $s'_r$. We will show that every strongly connected component of $\cA_r$ is strongly aperiodic.
	
	Let $R' \subset R$ be a terminal component. Because the action of $0$ is idempotent, there exists $s' = (s,0) \in R$ such that $\delta'(s',0) = s'$; fix a choice of such $s'$ and $s$. Denote $\cB = (R',s',\delta',\tau')$. For any $u \in \Sigma_k^*$ with $\delta(s,u) = s$ we have $[u]_k \in D_{\cA,s}$, and in particular $q \mid [u]_k$. Pick any $u$ as above, and $v \in 0\Sigma_k^*$ with $[u]_k = q[v]_k$. It follows from previous considerations that 
	
	$$
		\delta'\bra{s', v} = \bra{\delta\bra{s,u},0} = (s,0) = s'.
	$$
	Hence, $[u]_k/q \in D_{\cB,s'}$. Because $d_{\cA} = d_{\cA,s} = q$ is the greatest common divisor of the numbers $[u]_k$ for $u$ as above, $d_{\cB} = \gcd( D_{\cB,s'} ) = 1$. Hence, $\cB$ is strongly aperiodic by Proposition \ref{prop:very-primitive=>strongly-aperiodic}.
\end{proof}

\makeatletter{}\section{Pointwise convergence}\label{sec:Pointwise-convergence}

Along similar lines as in Section \ref{sec:Mean-Convergence}, we now finish the proof of Theorem \ref{thm:B} by means of reducing it to the following statement via Proposition \ref{prop:reduction:A}. 

\begin{proposition}\label{prop:cancellation-quantitative-A} 
	Let $a \colon \NN_0 \to \RR$ be a totally balanced $k$-automatic sequence. Then there exists a constant $c > 0$ such that 
	\begin{equation}
	\label{eq:760}
	\sup_{\alpha \in \RR} \abs{ \EE_{n < N} a(n) e(n\alpha) } \ll N^{-c} 
	\end{equation}
	as $N \to \infty$. 
\end{proposition}

\begin{proof}[Proof of Theorem \ref{thm:B}, assuming Proposition \ref{prop:cancellation-quantitative-A}]
Immediate by Proposition \ref{prop:reduction:A}.
\end{proof}

We devote the remainder of this section to proving this Proposition \ref{prop:cancellation-quantitative-A}. To begin with, we reduce to the case when $N$ is a power of $k$. 

\begin{lemma}\label{lem:wlog-N=k^L}
	Fix $k \geq 2$. Suppose that $a \colon \NN_0 \to \CC$ is a sequence such that 
	\begin{equation}
	\label{eq:950}
	\sup_{\alpha \in \RR} \abs{ \EE_{n < k^L} a(n) e(n\alpha) } \ll k^{-cL} 
	\end{equation}
	for some constant $c > 0$. Then it also holds that
	\begin{equation}
	\label{eq:951}
	\sup_{\alpha \in \RR} \abs{ \EE_{n < N} a(n) e(n\alpha) } \ll N^{-c'},
	\end{equation}
	where $c' = \frac{2}{3} c$.
\end{lemma}

\begin{proof}
	Let $N$ be a large number, and put $L := \ceil{\log_k N}$. 
 Observe 
\begin{equation}
	\label{eq:951a}
\sup_{\alpha \in \RR} \abs{ \EE_{n < N} a(n) e(n\alpha) } \ll \sup_{\alpha \in \RR} \abs{ \EE_{n < k^L} 1_{[N]}(n)a(n) e(n\alpha) }.
\end{equation}

	In order to use Fourier analysis, we replace $1_{[N]}$ by its smoothed version. Fix $\e > 0$ (independent of $N$, to be determined later), and take $M = k^{L- \floor{ \e L} }$. It will be convenient to do Fourier analysis in the finite group $\ZZ/k^L\ZZ$, with which the interval $[k^L]$ can naturally be identified. We will approximate $1_{[N]}$ with
	
	$$
		f(n) = k^{\floor{ \e L}} 1_{[N]} \ast 1_{[M]} (n) = k^{\floor{ \e L}} \EE_{m \in \ZZ/k^L\ZZ} 1_{[N]}(m) 1_{[M]} (n-m).
	$$
	
	Note that $ \EE_{n < k^L} \abs{ f(n) - 1_{[N]}(n) } \ll k^{-\e L}$ because $f(n) = 1_{[N]}(n)$ unless $n$ is within distance $M$ of $0$ or $N$ (modulo $k^L$). Hence,  by (\ref{eq:951a}) we may estimate  
	\begin{equation}
	\label{eq:952}
		\sup_{\alpha \in \RR} \abs{ \EE_{n < N} a(n) e(n\alpha) } \ll \sup_{\alpha \in \RR} \abs{ \EE_{n < k^L} a(n)f(n) e(n\alpha) } + O( k^{-\e L}).
	\end{equation}
	In order to estimate the right hand side of \eqref{eq:952}, expand 
	
	$$f(n) = \sum_{\xi \in \widehat{\ZZ/k^L\ZZ} } \hat f(\xi) e(n \xi),$$
	 where we identify $\widehat{\ZZ/k^L\ZZ}$ with $k^{-L} \ZZ \bmod 1$. We may now estimate, using \eqref{eq:950}:
	\begin{align*}
	 \abs{ \EE_{n < k^L} a(n)f(n) e(n\alpha) } & \leq \sum_{\xi } \abs{ \hat f(\xi) \EE_{n < k^L} a(n) e(n(\alpha +\xi))} 
	  \ll k^{-cL} \sum_{\xi} \abs{ \hat f(\xi) }.
	\end{align*}
	Using the Cauchy-Schwarz inequality and Parseval's equality we find
	\begin{align*}
	 \sum_{\xi} \abs{ \hat f(\xi) } & =  k^{\floor{ \e L}} \sum_{\xi} \abs{ \hat 1_{[N]}(\xi) } \abs{ \hat 1_{[M]}(\xi) } 
	 \\ & \leq k^{\floor{ \e L}} \bra{ \sum_{\xi} \abs{ \hat 1_{[N]}(\xi) }^2 }^{1/2} \bra{ \sum_{\xi} \abs{ \hat 1_{[M]}(\xi) }^2 }^{1/2}
	 \\ & = k^{\floor{ \e L}} \bra{ \EE_{n < k^L} 1_{[N]}(n)^2 }^{1/2} \bra{ \EE_{n < k^L} 1_{[M]}(n)^2 }^{1/2} 
	  \leq k^{ \e L/2}.
	\end{align*}
	Combining the above bounds, we find that 
	\begin{equation}
	\label{eq:953}
	\sup_{\alpha \in \RR} \abs{ \EE_{n < N} a(n) e(n\alpha) } \ll N^{-\e} + N^{\e/2 - c}.
	\end{equation}
	Choosing $\e = 2c/3$ we obtain the claim.
\end{proof}
\color{black}

We will derive Proposition \ref{prop:cancellation-quantitative-A} from the following more technical result.

\begin{proposition}\label{prop:cancellation-quantitative-B} \label{prop:canc-quant-Sigma} Fix $k \geq 2$. 
	Let $(S,\delta)$ be a strongly aperiodic $k$-automaton without output and initial state, and let $\cM$ be the set of totally balanced sequences produced by a $k$-automaton $(S,s_0,\delta,\tau)$ for some initial state $s_0 \in S$ and output $\tau \colon S \to \set{ z \in \CC }{ \abs{z} \leq 1}$. Then, there exists $c > 0$ (depending only on $(S,\delta)$) such that 

	\begin{equation}
		\sup_{a \in \cM} \sup_{\a \in \RR} \abs{ \EE_{ u \in \Sigma_k^L} a(u) e([u]_k \a) } \ll  k^{-cL} \text{ as } L \to \infty. \label{eq:433}
	\end{equation}	
\end{proposition}

\begin{proof}[Proof of Proposition \ref{prop:cancellation-quantitative-A} assuming Proposition \ref{prop:cancellation-quantitative-B}]

Let $a$ be a $\CC$-valued totally balanced $k$-automatic sequence, and let $N$ be a large integer. Or goal is to prove that the estimate \eqref{eq:760} holds. The reduction will consist of several steps. Let $\cA = (S,s_0,\delta,\tau)$ be a $k$-automaton producing $a$, and assume without loss of generality that $\cA$ ignores the leading $0$'s. 

\textit{Step 1.} It suffices to prove the assertion for all $\cA$ which are strongly connected and ignore leading $0$'s.
\begin{proof}
Take $L = \floor{ (\log_k N) / 2 }$. We may estimate
	\begin{equation}
	\label{eq:762-1}
	\sup_{\alpha \in \RR} \abs{ \EE_{n < N} a(n) e(n\alpha) } \ll \EE_{m < k^L} \sup_{\beta \in \RR} \abs{ \EE_{n < \floor{ N/k^L} } a(k^Ln + m) e(n\beta) } + O(N^{-1/2}).
	\end{equation}

	If $a$ is produced by $\cA = (S,\delta,s_0,\tau)$, then for any $m$ with $0 \leq m < k^L$, the sequence $a'_m$ given by $a'_m(n) = a(k^L n + m)$ is produced by the automaton $\cA'_m = (S,\delta,s_m',\tau)$, where $s'_m = \delta(s_0,(m)_k) $. The proportion of $m < k^L$ such that $s'_m$ fails to belong to a strongly connected component is $\ll k^{-c_1 L}$ for some constant $c_1 > 0$ (depending only on $(S,\delta)$). Note that for any $m$ with $0 \leq m < k^L$, the sequence $a'_m$ is totally balanced, the automaton $\cA'_m$ ignores leading $0$'s, and if $s'_m$ lies in a strongly connected component then $\cA'_m$ is strongly connected. 
	
	Suppose that we already know that \eqref{eq:760} holds with $c = c_{\mathrm{sc}}$ for those of $\cA'_m$ which are strongly connected. Applying this estimate where applicable (and estimating the remaining summands trivially by $1$) we conclude that \eqref{eq:760} holds for $\cA$ with $c = \min(c_1,c_{\mathrm{sc}}/2) > 0$.
\end{proof}
	
\textit{Step 2.}  It suffices to prove the assertion for all $\cA$ with the property that each strongly connected component of $\cA$ is strongly aperiodic and ignores the leading $0$'s (but $\cA$ may not be strongly connected).
\begin{proof}

By Proposition \ref{prop:aperiodic=>strongly-aperiodic}, there exist $q \in \NN$ and $k'$ (depending only on $(S,\delta)$) such that each of the sequences $a_r'(n) = a(qn + r)$ for $r$ with $0 \leq r < q$ is produced by a $k'$-automaton all of whose strongly connected components are aperiodic. Note that for $r$ with $0 \leq r < q$, the sequences $a_r'$ are again totally balanced. 

Using the analogue of \eqref{eq:762-1},
	\begin{equation}
	\label{eq:761}
	\sup_{\alpha \in \RR} \abs{ \EE_{n < N} a(n) e(n\alpha) } \ll \max_{ r \bmod q} \sup_{\beta \in \RR}\abs{ \EE_{n < \floor{ N/q} } a'_r(n) e(n\beta) } + O(1/N),
	\end{equation}
and the argument reminiscent of that in Step 1, we conclude that it is enough to prove the claim for automata with strongly aperiodic strongly connected components. 
\end{proof}

\textit{Step 3.} Without loss of generality, $\cA$ is strongly connected, strongly-antiperiodic and ignores the leading $0$'s.
\begin{proof}
	Assume, as we may, that each strongly connected component of $\cA$ is strongly aperiodic, and apply the same reduction as in Step 1. Note that, with notation as in Step 1, each of $\cA_m'$ has the property that each of its strongly connected components is strongly aperiodic. Hence, if \eqref{eq:760} holds for strongly connected and strongly aperiodic automata with constant $c = c_{\mathrm{sa}}$, then it holds for $\cA$ with $c = \min(c_1,c_{\mathrm{sa}}/2)$.
\end{proof}

\textit{Step 4.} Without loss of generality, $\cA$ is strongly connected, strongly-antiperiodic, ignores the leading $0$'s, and $N$ is a power of $k$.
\begin{proof}
	It follows from Lemma \ref{lem:wlog-N=k^L} that if \eqref{eq:760} holds for $N$ which are powers of $k$ with $c = c_{\mathrm{pow}}$, then it holds for general $N$ with $c = 2c_{\mathrm{pow}}/3$.
\end{proof}

The remainder of the claim follows directly from Proposition \ref{prop:cancellation-quantitative-B} applied to the partial automaton $(S,\delta)$.	
\end{proof}

\begin{proof}[Proof of Proposition \ref{prop:cancellation-quantitative-B}]
	Note that $\cM$ is compact (in the $\norm{\cdot}_\infty$ topology) and closed under the operations of taking kernels. 
	Denote for $a \in \cM$, $\alpha \in \RR$ and $L \geq 0$ the corresponding averages
	
	$$A(L,a,\a) = \abs{ \EE_{ u \in \Sigma_k^L} a(u) e( [u]_k \a) },$$
	and put $A(L,\a) = \sup_{a \in \cM} A(L,a,\a)$. By definition, $A(L,\a) \leq 1$.
		Recall from the proof of Proposition \ref{prop:cancellation-qualitative} that for each $l$ with $0 \leq l \leq L$ we have the recursive relation 
	
	\begin{align}
		\label{eq:450}		
		A(L,\a)
		 & \leq A(L-l,k^l\a) \cdot \sup_{a \in \cM} \sum_{b \in \cM }
		 \abs{ \EE_{ v \in \Sigma_k^l } 1_{S_{a,b}^{l}}(v)   e([v]_k \a)} ,
	\end{align}
	where the sets $S_{a,b}^l$ are given by 
	
	$$S_{a,b}^{l} = \set{ v \in \Sigma_k^l }{ a\bra{ u v} = b\bra{u} \text{ for all } u \in \Sigma_k^*}.$$ 
	Note that for given $a \in \cM$ and $l \geq 0$, the sets $S_{a,b}^{l}$ for $b \in \cM$ are a partition of $\Sigma_k^l$, and $S_{a,b}^l = \emptyset$ unless $b \in \cN_k(a)$. Hence, the sum in \eqref{eq:450} is really finite for each $a\in \cM$. 
	
	We aim to recursively exploit \eqref{eq:450} to obtain \eqref{eq:433}. In fact, depending on Diophantine properties of $\alpha$ and the value of $L$, we will use one of several estimates, with different values of $l$. Let $Q$ be the constant from Corollary \ref{cor:linear-dichotomy} applied with $r = \abs{S}$. Note that $Q$ may be replaced by any larger number; in particular we may assume without loss of generality that $Q \geq k$.

	\textit{Trivial estimate.} Let $\alpha \in \RR$ and $l$ with $0 \leq l \leq L$ be arbitrary. Then 
	
	$$A(L,\a)  \leq A(L-l,k^l\a).$$
	\begin{proof} 
	This is an immediate consequence of \eqref{eq:450}.
	\end{proof}	
	
	\textit{Minor arc estimate.} Let $\alpha \in \RR$ and  $0 \leq l \leq L$. Suppose that there exist no $p,q$ with $0 < q \leq Q$ such that $\abs{ \alpha - \frac{p}{q}} < \frac{Q}{k^l}$. Then $A(L,\alpha) \leq (1-\frac{1}{6 \abs{S}^2}) A(L-l,k^l \alpha)$.
	
	\begin{proof} 
	This is an immediate consequence of \eqref{eq:450} and Corollary \ref{cor:linear-dichotomy} applied for each $a \in \cM$ to the partition $\bar{S}_{a,b}^l = \set{[u]_k }{ u \in {S}_{a,b}^l}$ over $b \in \cN_k(a)$ of $[2^L]$ with $\leq \abs{S}$ parts.	
	\end{proof}	

	\textit{Major arc estimate.} There exist constants $c_{\mathrm{maj}}$ and $L_0$ (depending only on $\cA$) such that the following holds. Let $\alpha \in \RR$ and $l$ with $L_0  \leq l \leq L$. Suppose that there exist $p,q$ with $0 < p < q \leq Q$ such that $\abs{ \alpha - \frac{p}{q}} < 1/k^{l}$. Then
	
	 $$ A(L,\a) \leq k^{-c_{\mathrm{maj}} l} A(L-l,k^l \alpha).$$
	\begin{proof}
	Fix any choice of $a \in \cM$, $b \in \cN_k(a)$. In order to prove the claim, it will suffice to estimate the term 
\begin{equation}
	\label{eq:459}
	\abs{ \EE_{ v \in \Sigma_k^l } 1_{S_{a,b}^{l}}(v) e([v]_k \a)}
\end{equation}		
appearing in \eqref{eq:450}.

	Let $g \colon \Sigma_k^* \to \{0,1\}$ be the sequence given by
	
	$$
	g(v) = 
	\begin{cases}
		1, & \text{ if } a(uv) = b(u) \text{ for all } u \in \Sigma_k^*,\\
		0, & \text{ otherwise.}
	\end{cases}
	$$	
	
	Thus, for $v \in \Sigma_k^l$ we have $v \in S_{a,b}^l$ precisely when $g(v) = 1$. We further note that $g$ is generated by an automaton with states and transition function $(S,\delta)$. Indeed, suppose that $a$ is produced by the automaton $(S,s_0,\delta,\tau)$ for some choice of $s_0 \in S$ and $\tau \colon S \to \CC$. Let $S_1$ be the set of states $s_1 \in S$ such that $b$ is produced by the automaton $(S,s_1,\delta,\tau)$. Then, one immediately sees that $g(v) = 1$ precisely when $\delta(s_0,v) \in S_1$, so $g$ is produced by the automaton $(S,s_0,\delta,1_{S_1})$. 
	
	Note that $ g(u) e([u]_k p/q)$ is an automatic sequence, and it is balanced. To see this, note that we have
	
\begin{align*}
	\EE_{u \in \Sigma_k^m}  g(u) e([u]_k p/q) 
	& = \sum_{r \bmod q } e(pr/q) \EE_{\substack { u \in \Sigma_k^m\\ [u]_k \equiv r (q) }} g(u)
	\\ & \to \sum_{r \bmod q } e(pr/q) \sum_{s \in S_1} \pi(s;r(q)) 
	\\ & = \sum_{r \bmod q } e(pr/q) \sum_{s \in S_1} \pi(s) = 0, 
	\text{ as } m \to \infty,
\end{align*}
	where in the last step we use $p/q \not \in \ZZ$. It follows from Proposition \ref{prop:violet=>very-violet} that there exists a constant $c > 0$ such that 
	\begin{equation}
	\abs{ \EE_{u \in \Sigma_k^m}  g(u) e([u]_k p/q) } \ll k^{-cm}.
	\label{eq:460}
	\end{equation}
	A priori, the value of $c$ and of the implicit constant in \eqref{eq:460} depend on $g$; however only finitely many choices of $g$ are possible (since $g$ is fully determined by $S_1 \subset S$), hence we may choose these constants uniformly. By the same token (possibly after changing the value of $c$), we have for each $v \in \Sigma^*_k$ the bound 
	\begin{equation}
	\abs{ \EE_{u \in \Sigma_k^m}  g(vu) e([u]_k p/q) } \ll k^{-cm},
	\label{eq:460a}
	\end{equation}	
	where $c$ and the implicit constant do not depend on $v$.
	
	Using \eqref{eq:460a}, we may now etimate \eqref{eq:450}. Take $m = l(1-2c_{\mathrm{maj}})$. We find
	
	\begin{align*}
	\abs{ \EE_{u \in \Sigma_k^l}  g(u) e([u]_k \alpha) }
	 &\leq \EE_{v \in \Sigma_k^{l-m} } \abs{ \EE_{u \in \Sigma_k^m}  g(u) e([u]_k p/q) } + 
	O( \fpa{q\alpha} k^m )
	 \leq C k^{2 c_{\mathrm{maj}}l }
	\end{align*}	 
	for a constant $C$, assuming (as we may) that $c_{\mathrm{maj}}$ is chosen small enough with respect to $c$ that $c(1-2 c_{\mathrm{maj}}) > 2c_{\mathrm{maj}}$. Taking $L_0$ sufficiently large that $k^{c_{\mathrm{maj}} L_0} > C$ we conclude that 
	
	\begin{align*}
	\abs{ \EE_{u \in \Sigma_k^l}  g(u) e([u]_k \alpha) }
	 \leq k^{c_{\mathrm{maj}}l },
	\end{align*}
	and the claim follows from \eqref{eq:450}.
	\end{proof}
	
	\textit{Neighbourhood of $0$.} There exists a constant $c_{\mathrm{sml}} > 0$ such that for any $L \geq 0$ and $\alpha \in \RR$ it holds that 
	
	$$A(L,\a) \ll \fpa{\alpha}^{c_{\mathrm{sml}}},$$
	where the implicit constant depends only on $(S,\delta)$.
	\begin{proof}
	For any $m$ with $0 \leq m \leq L$ we have the estimate
	
	\begin{align}
	\label{eq:470}
		\abs{ A(L,\a) } & \leq \sup_{a \in \cM} \EE_{u \in \Sigma_k^{L-m} } \abs{ \EE_{v \in \Sigma_k^{m}} a(uv)e([v]_k \a) }
		 \leq  A(m,0) + O(\fpa{\alpha} k^{m}).
	\end{align}
		
	Proposition \ref{prop:violet=>very-violet} implies the bound $A(m,0) = O(k^{-cm})$ for a constant $c > 0$. Pick $m \sim \fpa{\alpha}^{-1/2}$; the required bound follows with $c_{\mathrm{sml}} = c/2$.
	\end{proof}
	
	We now proceed to prove \eqref{eq:433}. Take any $L \geq 0$ and $\alpha \in \RR$. If $\fpa{\alpha} < k^{-L/10}$ then \eqref{eq:433} follows from the estimate obtained for the neighbourhood of $0$, so suppose this is not the case. Thus, there is some $l$ with $0 \leq l \leq L/10$ such that $\fpa{k^l \alpha} \geq 1/k$, and by the trivial estimate we have 
	
	$$
		A(L,\alpha) \leq A(L-l, \fp{k^l \alpha}).
	$$
	If the estimate \eqref{eq:433} holds for $A(L-l, \fp{k^l \alpha})$, then it also holds for $A(L,\alpha)$ (with the constant $c$ smaller by the factor of $9/10$). Hence, replacing $\alpha$ with $\{k^l \alpha\}$ and $L$ with $L-l$, we may now assume that $\fpa{ \alpha } \geq 1/k$. We now combine the obtained estimates to obtain an inductive step.
		
	\textit{Combined estimate.} There exist constants $c_{\mathrm{cmb}} > 0$ and $L_1 \geq 0$ such that for any $\alpha \in \RR$ with $\fpa{\alpha} \geq 1/k$ and any $L \geq L_1$, there exists $l$ with $0 < l \leq L$ such that $A(L,\alpha) \leq k^{-c_{\mathrm{cmb}l}} A(L-l, k^l \alpha)$ and if $l \neq L$ then $\fpa{k^l \alpha} \geq 1/k$. 
	\begin{proof}
To begin with, we branch off into cases depending on the length of the longest string of $0$'s in the first $L_1$ digits of $\alpha$ base $k$.
	
	Suppose that there is a string of $\geq L_1/2$ consecutive $0$'s in the initial $L_1$ digits of $\alpha$. This means that there is some $l_1 \leq L_1/2$ and a digit $j$ with $0 \leq j < k$ such that $\fpa{ k^{l_1} \alpha - \frac{j}{k} } \leq k^{-L_1/2} $. Using the trivial estimate we have $A(L,\alpha) \leq A(L-l_1,k^{l_1}\alpha)$. Pick the largest $l_2 \leq L$ such that $\fpa{ k^{l_1} \alpha - \frac{j}{k} } \leq k^{-l_2}$; note that $l_2 \geq L_1/2$. Recall that $Q \geq k$. Hence, the major arc estimate can be used to estimate
	
	$$
	A(L-l_1,k^{l_1}\alpha) \leq k^{-c_{\mathrm{maj} }l_2} A(L-l_1-l_2, k^{l_1+l_2} \alpha).
	$$
	By the choice of $l_2$ we have, $\fpa{ k^{l_1+l_2} \alpha} \geq 1/k$ or $l_2 = L-l_1$. Hence, the claim holds with $l = l_1+l_2$ (as long as $c_{\mathrm{cmb}} \leq c_{\mathrm{maj}}/2$).
	
	Suppose now that there is no string of $0$'s of length $\geq L_1/2$. Again, there are two cases to consider, depending on whether there exist $p,q$ with $0 \leq p < q \leq Q$ such that $\fpa{ \alpha - \frac{p}{q}} \leq Q/k^{L_1/2}$.
	
	If no, then  we may apply the minor arc estimate with $l_1 = L_1/2$ to obtain
	\begin{equation}
	\label{eq:470a}
		A(L,\alpha) \leq k^{- c_{\mathrm{cmb}} L_1 }  A(L-l_1,k^{l_1} \alpha),
	\end{equation}
	(as long as $c_{\mathrm{cmb}}$ is small enough that $1-\frac{1}{6\abs{S}^2} \leq k^{- c_{\mathrm{cmb}} L_1 }$).
	
	Otherwise, the major arc estimate is applicable with $l_1 = \floor{ L_1/2 - \log_k Q} \geq L_1/3 \geq L_0$, provided that $L_1$ is chosen large enough that $L_1 \geq \max( 6 \ceil{ \log_k Q }, 3 L_0)$. We obtain the same estimate
	\begin{equation}
	\label{eq:470b}
		A(L,\alpha) \leq k^{- c_{\mathrm{cmb}} L_1 }  A(L-l_1,k^{l_1} \alpha)
	\end{equation}
	(as long as $c_{\mathrm{cmb}}$ is small enough that $  c_{\mathrm{cmb}} \leq c_{\mathrm{maj}}/3$). 
	
	In either case, let $l_2$ with $0 \leq l_2 \leq L_1 - l_1$ be the least integer such that $\fpa{k^{l_1+l_2}\alpha} \geq 1/k$. Such an $l_2$ exists, because $\alpha$ is assumed to have at least one non-zero digit at positions between $l_1+1$ and $l_1 + L_1/2 \leq L_1$. Applying the trivial bound with $l_2$ to the right hand side of \eqref{eq:470a} or \eqref{eq:470b} respectively, we thus 
	
	\begin{equation*}
		A(L,\alpha) \leq k^{- c_{\mathrm{cmb}} L_1 }  A(L-l_1-l_2,k^{l_1+l_2} \alpha),
	\end{equation*}
	so the claim holds with $l = l_1+l_2$.
	\end{proof}	
	
	Iterating the combined estimate we have just obtained gives \eqref{eq:433} (with $c = c_{\mathrm{cmb}}$) by a simple inductive argument.
\end{proof}

\makeatletter{}\section{Invertible sequences}\label{sec:Invertible}
\newcommand{\id}{\mathrm{id}}
In this section, we deal with \emph{invertible} automatic sequences, and show that, in a quantitative sense, they cannot correlate with polynomial phases. As discussed in Section \ref{sec:outline}, once this is accomplished, Theorem \ref{thm:D} will follow.

A $k$-automatic sequence $a \colon \NN_0 \to \Omega$ is \emph{invertible} if it is generated by an automaton $\cA = (S,s_0,\delta,\tau)$ such that for each $j \in \Sigma_k$, the map $\delta(\cdot,j) \colon S \to S$ is invertible \cite{DrmotaMorgenbesser-2012}. A \emph{generalised Thue--Morse sequence} taking values in a finite group $G$ is a $k$-automatic sequence $g \colon \Sigma_k^* \to G$ such that $g(uv) = g(u)g(v)$ for any $u,v \in \Sigma_k^*$, and $g(0) = \id_G$. Note that $g(n)$ is then characterised by $g(1),\dots,g(k-1)$. Invertible sequences are now precisely the ones of the form $a(n) = \pi( g(n) )$, where $g(n) \in G$ is a generalised Thue--Morse sequence, and $\pi \colon G \to \Omega$ is any function (see \cite{DrmotaMorgenbesser-2012} for further discussion).

\begin{lemma}\label{lem:invert-prod} If $a_1,a_2$ are invertible $k$-automatic sequences, then so is $(a_1,a_2)$. In particular, the family of $\CC$-valued invertible $k$-automatic sequences is a ring.
\end{lemma}
\begin{proof}
	Let $a_i$ be produced by $k$-automata $\cA_i = (S_i, s_{i,0}, \delta_i, \tau_i)$. Then $a = (a_1,a_2)$ is produced by the $k$-automaton $\cA = (S,s_0,\delta,\tau)$ where $S = S_1 \times S_2$, $s_0 = (s_{0,1}, s_{0,2})$, $\delta = \delta_1 \times \delta_2$ (i.e.\ $\delta\bra{ (s_1,s_2), j} = \bra{ \delta(s_1,j), \delta(s_2,j)}$) and $\tau(s_1,s_2) = \bra{ \tau_1(s_1), \tau_2(s_2)}$. It remains to note that a Cartesian product of invertible maps is invertible.
\end{proof}

The following is a direct consequence of \cite[Theorem 3]{DrmotaMorgenbesser-2012}. A sequence $a(n)$ is $d$-periodic if $a(n+d) = a(n)$ for all $n$; we do not require that $d$ should be the least period.

\begin{proposition}\label{prop:invert-decomp}
	Fix $k \geq 2$. Let $a \colon \NN_0 \to \CC$ be an invertible $k$-automatic sequence. Then, there exists a decomposition $$a(n) = a_{\mathrm{per}}(n) + a_{\mathrm{bal}}(n),$$
	where $a_{\mathrm{per}}(n)$ is $(k-1)$-periodic and $a_{\mathrm{bal}}(n)$ is totally balanced (and invertible).
\end{proposition}

Theorem \ref{thm:D} reduces to the following being the main result of this section.

\begin{proposition}\label{prop:cancel-quant-poly}
	Fix $k \geq 2$ and $d \in \NN$. Let $a \colon \NN_0 \to \Omega \subset \CC$ be a totally balanced invertible $k$-automatic sequence. Then there exists a constant $c > 0$ such that
	\begin{equation}
	\label{eq:562} 
	\sup_{ \substack{ p(x) \in \RR[x] \\ \deg p \leq d }} 
	\abs{ \EE_{n < N} a(n) e( p(n) ) } \ll N^{-c}.
	\end{equation}
\end{proposition}

\begin{proof}[Proof of Theorem \ref{thm:D}, assuming Proposition \ref{prop:cancel-quant-poly}]
Immediate by Proposition \ref{prop:reduction:A}.
\end{proof}

The following technical lemma will be crucial in the argument. Recall that $\Delta_h p(n) := p(n+h) - p(n)$.

\begin{lemma}\label{lem:poly-decomp}
	Fix $d \geq 2$. Let $p(x) \in \RR[x]$ be a polynomial of degree $\leq d$, and suppose that for some $M,N,H \in \NN$ and $\e > 0$ there exist $m,h \in \NN_0$ with $1 \leq h < H$ and $m < M$ such that
\begin{equation}
		\abs{ \EE_{n < N} e(\Delta_h p(M n + m)) } \geq \e.
		\label{eq:942}
\end{equation}
	Then, there exists a decomposition $p(x) = \tilde p(x) + r(x)$, where $e(r(n))$ is periodic with period $\ll H^{O(1)} M^{O(1)}$ and $\tilde p(n)$ is approximately linear in the sense that for each $n_0 \in \ZZ$ there exist $\beta_0(n_0), \beta_1(n_0) \in \RR$ such that for all $n \in \ZZ$ we have 
	\begin{equation}
	\label{eq:944}
	\tilde p(n_0 + n) = \beta_0(n) + \beta_1(n) n + O\bra{\frac{n^2}{N}\bra{1+ \frac{\abs{n_0} + \abs{n} }{N}}^{d-2} \frac{H^{O(1)} M^{O(1)}}{\e^{O(1)} } },
	\end{equation}
where the implicit constants depend on $d$ only.
	In fact, $\beta_0(n) = \tilde p(n)$ and $\beta_1(n) = \Delta_1 \tilde p(n)$.
\end{lemma}
\begin{proof}

Fix a choice of $h$ and $m$. Let us expand 
$$p(n) = \sum_{i=0}^d \alpha_i \binom{n}{i},\quad \text{and} \quad 
\Delta_h p (M n + m) = \sum_{i=0}^{d-1} \gamma_i \binom{n}{i}.$$

 It follows from the Quantitative Weyl Theorem (see e.g.\ \cite[Prop.~4.4]{GreenTao-2012}) and \eqref{eq:942} that there exists $0 < Q \ll 1/\e^{O(1)}$ such that 
$$\fpa{ Q \gamma_i } \ll 1/(N^{i} \e^{O(1)})$$
for all $i$ with $1 \leq i \leq d-1$. Note that $\gamma_i$ and $\alpha_i$ are related by a linear relation of the form
\begin{equation}
	\label{eq:655}
	\begin{bmatrix}
	\gamma_{d-1} \\
	\gamma_{d-2} \\
	\gamma_{d-3} \\	
	\vdots \\
	\gamma^{h,m}_0
	\end{bmatrix} = 
	\begin{bmatrix}
	M^{d} h & 0 & 0 & \dots & 0 \\
	\ast & M^{d-1} h & 0 & \dots & 0 \\
	\ast & \ast & M^{d-2} h & \dots & 0 \\
	\vdots & \dots & \vdots & \ddots & 0  \\
	\ast & \ast & \ast & \dots & Mh \\	
	\end{bmatrix} \cdot
	\begin{bmatrix}
	\alpha_{d} \\
	\alpha_{d-1} \\
	\alpha_{d-2} \\
	\vdots \\
	\alpha_1
	\end{bmatrix},
\end{equation}
where all the hidden coefficients are polynomials in $M,h$ and $m$ with total degree bounded in terms of $d$.

It follows from {the formula for the inverse matrix} that the numbers $M^{\binom{d}{2}} h^d \alpha_i$ ($1 \leq i \leq d$) are linear combinations of $\gamma_j$ ($0 \leq j < d$) with integer coefficients of magnitude $\ll H^{O(1)} M^{O(1)}$. 

Let us put $Q' = M^{\binom{d}{2}} h^d Q$, so that 
$$
	\fpa{ Q' \alpha_i } \ll H^{O(1)} M^{O(1)}/(N^{i-1} \e^{O(1)})
$$	
for each $i$ with $2 \leq i \leq d$. In other words, we may choose $K \ll H^{O(1)} M^{O(1)}/\e^{O(1)}$ such that for each $i$ with $2 \leq i \leq d$ we may decompose 
\begin{equation}\label{eq:943}
\alpha_i = \tilde\alpha_i + a_i/Q'
\end{equation}
 with $\norm{\alpha_i} \leq K/N^{i-1}$ and $a_i \in \ZZ$. For $i \in \{0,1\}$, declare $\tilde \alpha_i = \alpha_i$ and $a_i = 0$ so that \eqref{eq:943} continues to hold. We now put $\tilde p(x) = \sum_{i=0}^d \tilde \alpha_i \binom{n}{i}$ and $r(x) = \sum_{i=0}^d \bra{a_i/Q'} \binom{n}{i}$.
 
With these definitions, $e(r(n))$ is evidently periodic with period $Q'' = d! Q' \ll M^{O(1)}H^{O(1)}$. As for $\tilde p(n)$, for any $n \in \ZZ$ find that 
\begin{equation}
	\fpa{ \Delta_1^2 \tilde p (n) } = \fpa{ \sum_{i=2}^{d-2} \tilde \alpha_i \binom{n}{i-2} } \ll \bra{ 1+ \frac{\abs{n}}{N} }^{d-2} K/N. 
	\label{eq:945}
\end{equation}
	Fix a choice of $n_0$, and put $\beta_0 = \tilde p(n_0)$ and $\beta_1 = \Delta_1 \tilde p(n_0)$. A simple inductive argument shows that for all $n$ we have 
$$	\Delta_1 \tilde p(n_0 + n) = \Delta_1 \tilde p(n_0) + \sum_{j=0}^{n-1} \Delta_1^2 p(n_0+j) = \beta_1 + O\bra{ n \bra{1+\frac{\abs{n} + \abs{n_0}}{N}}^{d-2} K/N}.$$
By a similar reasoning, we obtain 
$$	\tilde p(n_0 + n) = \beta_0 + n\beta_1 + O\bra{ n^2 \bra{1 + \frac{\abs{n} + \abs{n_0}}{N}}^{d-2} K/N}. \qedhere$$
\end{proof}

\begin{remark*}
	We are particularly interested in the regime where $H,M,1/\e = N^{O(\delta)}$ for a small constant 	$\delta>0$ and $n_0,n = O(N)$. Then, the period of $e(r(n))$ is $N^{O(\delta)}$ and \eqref{eq:944} simplifies to
	$$\tilde p(n_0 + n) = \beta_0 + \beta_1 n + O(n^2 N^{-1 + O(\delta)}).$$
\end{remark*}

Another ingredient which we will need is some observations about co-kernels of invertible sequences. Pick a $k$-automatic invertible sequence $a \colon \Sigma_k^* \to \Omega$. Recall that $a$ has a representation $a = \pi \circ g$ where $g \colon \Sigma_k^* \to G$ is a generalised Thue--Morse sequence taking values in a group $G$ and $\pi \colon G \to \Omega$. In particular, $a$ ignores the leading $0$'s. There is a natural choice of an automaton $\cA = (S,s_0,\delta,\tau)$ producing $a$, namely $S = G$, $s_0 = h$, $\delta(s,j) = g(j)s$, $\tau(g) = \pi(gh^{-1})$, where $h \in G$ (one may take $h = \id_G$ for concreteness). Conversely, any automaton of the above form is invertible. 

It is clear from the above description that if $a \colon \Sigma_k^* \to \Omega$ is an invertible $k$-automatic sequence, then any sequence $b \in \cN_k(a) \cup \cN_k'(a)$ is also invertible. 

Note also that any invertible sequence automatically ignores the leading $0$'s. Hence, we may freely identify invertible sequences $\NN_0 \to \Omega$ with invertible sequences $\Sigma_k^* \to \Omega$. In particular, if $a \colon \NN_0 \to \Omega$ is an invertible sequence them it makes sense to consider $\cN_k'(a)$, and any $b \in \cN_k'(a)$ may naturally be viewed as a sequence $\NN_0 \to \Omega$. 

\begin{lemma}
	Let $a \colon \NN_0 \to \CC$ be an invertible and totally balanced $k$-automatic sequence, and let $b \in \cN_k'(a)$. Then $b$ is totally balanced. 
\end{lemma}
\begin{proof}
	Fix $q \in \NN$; we will show that any $b \in \cN_k'(a)$ does not correlate with $q$-periodic sequences. Denote
	
	$$ A(L) = \max_{b \in \cN'_k(a)} \max_{0 \leq r < q} \abs{ \EE_{n < k^L} b(n) e(nr/q)}.$$
	For any $N$, we may partition $[N]$ into disjoint intervals of the form $I_{m,l} = [m k^l, (m+1)k^l)$ where each value of $l$ appears at most $k-1$ times. It follows that for any ${b \in \cN'_k(a)}$ and any $r$ with ${0 \leq r < q}$ we have
	
	$$ \abs{ \EE_{n < k^L} b(n) e(nr/q)} \ll \sum_{l \leq \log_k N} \frac{k^l}{N} A(l).$$
	Hence, it will suffice to prove that $A(L) \to 0$ as $L \to \infty$. Fix $b \in \cN_k'(a)$. There exists some $v \in \Sigma_k^*$ such that $b(u) = a(vu)$ for all $u \in \Sigma_k^*$. Hence
	
	$$
	 \max_{0 \leq r < q} \abs{ \EE_{n < k^L} b(n) e(nr/q)} =
	 \max_{0 \leq r < q} \frac{1}{K^L} \abs{ \sum_{ n = [v]_k k^L }^{([v]_k+1) k^L - 1 } a(n) e(nr/q)} \to 0
	$$
	as $L \to \infty$, because $a$ is totally balanced.	
\end{proof}

\begin{proof}[Proof of Proposition \ref{prop:cancel-quant-poly}]

{We proceed by induction on $d$. The case $d=1$ follows from Proposition \ref{prop:cancellation-quantitative-A}.  Fix $d \in \NN$ and assume that the claim holds for $d-1$. 

Using the analogue of Lemma \ref{lem:wlog-N=k^L}, we may assume that $N=k^L$ for some $L$. (In fact, we may assume also that $L$ is divisible by a specified large integer $D$ in order to ensure later in the argument that various small multiples of $L$ are also integers.)

Let $p \in \RR[x]$ with $\deg p = d$ and assume  without loss of generality that $p(0) = 0$.  The implicit (and explicit) constants below are independent of $p$.}
	
	We begin by using the van der Corput inequality \ref{lem:vdCorput} with $H = k^{\delta L}$, where $\delta > 0$ is a small constant to be determined later, and we assume for the sake of clarity that $\delta L$ is an integer. We obtain
	\begin{equation}
	\label{eq:563} 
	\abs{ \EE_{n < N} a(n) e( p(n) ) }^2 \ll
	\EE_{h < H}\abs{ \EE_{n < N} e\bra{\Delta_h p(n)} a(n) \bar{a}(n+h) }  + O(H/N).
	\end{equation}

The inner average on the right hand side of \eqref{eq:563} is again an average of a $k$-automatic sequence weighted by a polynomial phase (cf.\ Corollary \ref{prop:cancellation-qualitative-polynomial}). However, all of the sequences $a(n) \bar{a}(n+h)$ may potentially be different, which makes tracing error terms difficult. To overcome this, introduce an additional average. Let $M = k^{2 \delta L}$. We have for any $h$ with $h < H$ that

\begin{align*}
 \abs{ \EE_{n < N} e\bra{\Delta_h p(n)} a(n) \bar{a}(n+h) } 
 & \leq  \EE_{m < M} \abs{ \EE_{n < N/M} e\bra{q_{h,m}(n)} a(Mn + m) \bar{a}(Mn+ m + h) }
\\ & \leq \EE_{m < M}\max_{b,b' \in \cN_k(a)} \abs{ \EE_{n < N/M} e\bra{q_{h,m}(n)} b(n) \bar{b}'(n) } + O(H/M)
\end{align*}
where $q_{h,m}(n) = \Delta_h p(Mn + m)$. In the last line we use the fact that $a(Mn + m + h) = b'(n)$ for some $b' \in \cN_k(a)$, provided that $m < M - h$. Importantly, the maximum is taken over a finite set.

Since each of the sequences $ b(n) \bar{b}'(n)$ appearing above is invertible by Lemma \ref{lem:invert-prod},  we may decompose it into a periodic and aperiodic part by Proposition \ref{prop:invert-decomp}. Applying the inductive assumption to the aperiodic part and removing the periodic component by splitting the average further, we conclude that

\begin{align*}
 \abs{ \EE_{n < N} e(\Delta_h p(n)) a(n) \bar{a}(n+h) } 
 & \ll \max_{\substack{ 0 \leq m < M' }} \abs{ \EE_{n < N/M'} e\bra{q_{h,m}'(n)}} + O((N/M)^{-c_1} + H/M)
\end{align*}
where $M' = (k-1)M$, $q'_{h,m}(n) = \Delta_h p(M'n + m)$, and $c_1 > 0$ is a constant depending only on $a$, originating from the inductive assumption. Plugging this back into \eqref{eq:563}, we conclude that the claim \eqref{eq:562} holds with $c = \frac{1}{2} \min(c_1,\delta)$, unless for many $h$ (namely, $\gg N^{\delta/2}$ values of $h$) there exists $m$ with $0 \leq m < M'$ such that

\begin{equation}
	\label{eq:654}
 \abs{ \EE_{n < N/M'} e\bra{q_{h,m}'(n) }} \gg N^{-\delta}.
\end{equation}

Suppose now that \eqref{eq:654} holds for some $h \neq 0$ and $m$. In this case, we may apply Lemma \ref{lem:poly-decomp} to obtain a decomposition $p(n) = \tilde p(n) + r(n)$ where $e(r(n))$ is periodic with period $Q = N^{O(\delta)}$ and $\tilde p$ is approximately linear in the sense that
$$\tilde p(n_0+n) = \beta_0(n_0) + n \beta_1(n_0) + O(n^2 N^{-1+O(\delta)})$$
for any $\abs{n},\abs{n_0} < N$, where $\beta_0(n_0),\beta_1(n_0) \in \RR$.

Assume for the sake of clarity that $L/3$ is an integer. We may then estimate
\begin{equation}
	\abs{ \EE_{n < N} a(n) e(p(n)) } 
	\leq \max_{ v \in \Sigma_k^{2L/3} } \abs{ \EE_{ u \in \Sigma_k^{L/3} } 
	a( vu ) e(\tilde p( [vu]_k )) e( r( [vu]_k )) },
	\label{eq:567}
\end{equation}
Note that $a(vu) = b(u)$ for some $b \in \cN_k'(a)$, dependent on $v$. Next, we observe that $e(\tilde p( [vu]_k )) = e(\beta_0) e(\beta_1 [u]_k) + O(N^{-1/3 + O(\delta)})$ for $\beta_0,\beta_1 \in \RR$, dependent on $v$. Finally, by periodicity, $e( r( [vu]_k )) = \sum_{j=0}^{Q-1} w_j e([u]_k j/Q)$, where $\abs{w_j} \leq 1$ for all $j$. Plugging these into \eqref{eq:567} we conclude that 
\begin{equation}
	\abs{ \EE_{n < N} a(n) e(p(n)) } 
	\leq Q \max_{b \in \cN_k'(a)} \sup_{\alpha \in \RR} \abs{ \EE_{ u \in \Sigma_k^{L/3} } 
	b( u ) e(\alpha [u]_k ) } + O(N^{-1/3 + O(\delta)}),
	\label{eq:568}
\end{equation}

Applying Proposition \ref{prop:cancellation-quantitative-B} (or \ref{prop:cancellation-quantitative-A}) to bound the averages on the right hand side of \eqref{eq:568}, we conclude that 
\begin{equation}
	\abs{ \EE_{n < N} a(n) e(p(n)) } \ll N^{-c_1/3 + O(\delta)} + O(N^{-1/3 + O(\delta)}),
	\label{eq:569}
\end{equation}
for a constant $c_1 > 0$. Assuming (as we may) that $\delta$ is small enough that the term $O(\delta)$ above is less than $\min(1,c_1)/3$, this finishes the argument.
\end{proof}

Having finished the proof of Theorem \ref{thm:D}, it is natural to ask about the extent to which results such as Proposition \ref{prop:cancel-quant-poly} can be generalised to other classes of automatic sequences. As the following example shows, even for relatively simple sequences, direct generalisations fail.

\begin{example}\label{ex:TM-shifted-product}
	Let $t(n) = (-1)^{s_2(n)}$ be the Thue--Morse sequence. 
	
	 Then $t(n)t(n+1) = (-1)^{\nu_2(n)+1}$ where $\nu_2(n)$ denotes the highest power of $2$ dividing $n$. Hence, $t(n)t(n+1)$ does not admit a decomposition into a periodic and aperiodic part. Note also that $\cN_2( t(n+1) ) = \{ \pm t(n), t(n+1) \}$. 
	 
	 In particular, the methods in the proof of Proposition \ref{prop:cancel-quant-poly} cannot be directly applied to the sequence $t(n+1)$. 
	 
\end{example}

\begin{remark*}
	The only point where the proof of Proposition \ref{prop:cancel-quant-poly} essentially uses the assumption that the sequence $a(n)$ is invertible is to ensure that any sequence $b(n)$ in the (multiplicative) group generated by $\cN_k(a)$ has a decomposition $b(n) = b_{\mathrm{per}}(n) + b_{\mathrm{bal}}(n)$ into a periodic and totally balanced part. We believe that similar results can be proven for any class of automatic sequences having this property.
\end{remark*} 

\makeatletter{} 

\bibliographystyle{alphaabbr}
\bibliography{bibliography}

\begin{thebibliography}{EAKPLdlR17}

\bibitem[AM15]{AssaniMoore}
I.~Assani and R.~Moore.
\newblock A good universal weight for multiple recurrence averages with
  commuting transformations in norm.
\newblock Preprint, available at https://arxiv.org/abs/1506.06730, 2015.

\bibitem[AP14]{AssaniPresser}
I.~Assani and K.~Presser.
\newblock A survey of the return times theorem.
\newblock In {\em Ergodic theory and dynamical systems}, De Gruyter Proc.
  Math., pages 19--58. De Gruyter, Berlin, 2014.

\bibitem[AS92]{AlloucheShallit-1992}
J.-P. Allouche and J.~Shallit.
\newblock The ring of {$k$}-regular sequences.
\newblock {\em Theoret. Comput. Sci.}, 98(2):163--197, 1992.

\bibitem[AS03]{AlloucheShallit-book}
J.-P. Allouche and J.~Shallit.
\newblock {\em Automatic sequences. Theory, Applications, Generalizations.}
\newblock Cambridge University Press, Cambridge, 2003.

\bibitem[Ass98]{Assani98}
I.~Assani.
\newblock A weighted pointwise ergodic theorem.
\newblock {\em Ann. Inst. H. Poincar\'e Probab. Statist.}, 34(1):139--150,
  1998.

\bibitem[Ass03]{Assani-book}
I.~Assani.
\newblock {\em Wiener {W}intner {E}rgodic {T}heorems}.
\newblock World Scientific Publishing Co., Inc., River Edge, NJ, 2003.

\bibitem[BCH14]{BellCoonsHare-2014}
J.~P. Bell, M.~Coons, and K.~G. Hare.
\newblock The minimal growth of a {$k$}-regular sequence.
\newblock {\em Bull. Aust. Math. Soc.}, 90(2):195--203, 2014.

\bibitem[BCH16]{BellCoonsHare-2016}
J.~P. Bell, M.~Coons, and K.~G. Hare.
\newblock Growth degree classification for finitely generated semigroups of
  integer matrices.
\newblock {\em Semigroup Forum}, 92(1):23--44, 2016.

\bibitem[BFKO89]{BFKO}
J.~Bourgain, H.~Furstenberg, Y.~Katznelson, and D.~S. Ornstein.
\newblock Appendix on return-time sequences.
\newblock {\em Publ. Math., Inst. Hautes \'Etudes Sci.}, (69):42--45, 1989.

\bibitem[BL85]{BellowLosert85}
A.~Bellow and V.~Losert.
\newblock The weighted pointwise ergodic theorem and the individual ergodic
  theorem along subsequences.
\newblock {\em Trans. Amer. Math. Soc.}, 288(1):307--345, 1985.

\bibitem[BLRT02]{BLRT}
D.~Berend, M.~Lin, J.~Rosenblatt, and A.~Tempelman.
\newblock Modulated and subsequential ergodic theorems in {H}ilbert and
  {B}anach spaces.
\newblock {\em Ergodic Theory Dynam. Systems}, 22(6):1653--1665, 2002.

\bibitem[BM10]{BuczolichMauldin}
Z.~Buczolich and R.~D. Mauldin.
\newblock Divergent square averages.
\newblock {\em Ann. of Math. (2)}, 171(3):1479--1530, 2010.

\bibitem[Bou89]{Bourgain89}
J.~Bourgain.
\newblock Pointwise ergodic theorems for arithmetic sets.
\newblock {\em Publ. Math., Inst. Hautes \'Etudes Sci.}, (69):5--45, 1989.
\newblock With an appendix by the author, Harry Furstenberg, Yitzhak Katznelson
  and Donald S. Ornstein.

\bibitem[Chu09]{Chu09}
Q.~Chu.
\newblock Convergence of weighted polynomial multiple ergodic averages.
\newblock {\em Proc. Amer. Math. Soc.}, 137(4):1363--1369, 2009.

\bibitem[CLO98]{CLO}
D.~C\"omez, M.~Lin, and J.~Olsen.
\newblock Weighted ergodic theorems for mean ergodic {$L_1$}-contractions.
\newblock {\em Trans. Amer. Math. Soc.}, 350(1):101--117, 1998.

\bibitem[CW17]{CunyWeber}
C.~Cuny and M.~Weber.
\newblock Ergodic theorems with arithmetical weights.
\newblock {\em Israel J. Math.}, 217(1):139--180, 2017.

\bibitem[DM12]{DrmotaMorgenbesser-2012}
M.~Drmota and J.~F. Morgenbesser.
\newblock Generalized {T}hue-{M}orse sequences of squares.
\newblock {\em Israel J. Math.}, 190:157--193, 2012.

\bibitem[DM17]{DrappeauMullner-2017+}
S.~Drappeau and C.~M{\"{u}}llner.
\newblock Exponential sums with automatic sequences.
\newblock Preprint, available at https://arxiv.org/abs/1710.01091, 2017.

\bibitem[Dum13]{Dumas-2013}
P.~Dumas.
\newblock Joint spectral radius, dilation equations, and asymptotic behavior of
  radix-rational sequences.
\newblock {\em Linear Algebra Appl.}, 438(5):2107--2126, 2013.

\bibitem[Dum14]{Dumas-2014}
P.~Dumas.
\newblock Asymptotic expansions for linear homogeneous divide-and-conquer
  recurrences: algebraic and analytic approaches collated.
\newblock {\em Theoret. Comput. Sci.}, 548:25--53, 2014.

\bibitem[EAKPLdlR17]{AHKLR}
E.~H. El~Abdalaoui, J.~Ku\l{}aga-Przymus, M.~Lema\'nczyk, and T.~de~la Rue.
\newblock The {C}howla and the {S}arnak conjectures from ergodic theory point
  of view.
\newblock {\em Discrete Contin. Dyn. Syst.}, 37(6):2899--2944, 2017.

\bibitem[EFHN15]{EFHN15}
T.~Eisner, B.~Farkas, M.~Haase, and R.~Nagel.
\newblock {\em Operator theoretic aspects of ergodic theory}, volume 272 of
  {\em Graduate Texts in Mathematics}.
\newblock Springer, Cham, 2015.

\bibitem[Eis13]{Eisner13}
T.~Eisner.
\newblock Linear sequences and weighted ergodic theorems.
\newblock {\em Abstr. Appl. Anal.}, Art. ID 815726, 5 pages, 2013.

\bibitem[Eis15]{Eisner15}
T.~Eisner.
\newblock A polynomial version of {S}arnak's conjecture.
\newblock {\em C. R. Math. Acad. Sci. Paris}, 353(7):569--572, 2015.

\bibitem[EK16]{EisnerKrause}
T.~Eisner and B.~Krause.
\newblock ({U}niform) convergence of twisted ergodic averages.
\newblock {\em Ergodic Theory Dynam. Systems}, 36(7):2172--2202, 2016.

\bibitem[EW11]{EinsiedlerWard}
M.~Einsiedler and T.~Ward.
\newblock {\em Ergodic theory with a view towards number theory}, volume 259 of
  {\em Graduate Texts in Mathematics}.
\newblock Springer-Verlag London, Ltd., London, 2011.

\bibitem[EZK13]{EisnerZorin}
T.~Eisner and P.~Zorin-Kranich.
\newblock Uniformity in the {W}iener-{W}intner theorem for nilsequences.
\newblock {\em Discrete Contin. Dyn. Syst.}, 33(8):3497--3516, 2013.

\bibitem[Fan17]{Fan-2017+}
A.-H. Fan.
\newblock Weighted birkhoff ergodic theorem with oscillating weights.
\newblock {\em Ergodic Theory and Dynamical Systems}, page 1–15, 2017.

\bibitem[Fra06]{Frantz}
N.~Frantzikinakis.
\newblock Uniformity in the polynomial {W}iener-{W}intner theorem.
\newblock {\em Ergodic Theory Dynam. Systems}, 26(4):1061--1071, 2006.

\bibitem[Gel68]{Gelfond-1968}
A.~O. Gel'fond.
\newblock Sur les nombres qui ont des propri\'et\'es additives et
  multiplicatives donn\'ees.
\newblock {\em Acta Arith.}, 13:259--265, 1967/1968.

\bibitem[GT12]{GreenTao-2012}
B.~Green and T.~Tao.
\newblock The quantitative behaviour of polynomial orbits on nilmanifolds.
\newblock {\em Ann. of Math. (2)}, 175(2):465--540, 2012.

\bibitem[HK09]{HostKra09}
B.~Host and B.~Kra.
\newblock Uniformity seminorms on {$\ell^\infty$} and applications.
\newblock {\em J. Anal. Math.}, 108:219--276, 2009.

\bibitem[Kon17]{Konieczny-2017+}
J.~Konieczny.
\newblock {Gowers norms for the Thue-Morse and Rudin-Shapiro sequences}.
\newblock Preprint, available at https://arxiv.org/abs/1611.09985, 2017.

\bibitem[KZK15]{KrauseZorin}
B.~Krause and P.~Zorin-Kranich.
\newblock A random pointwise ergodic theorem with {H}ardy field weights.
\newblock {\em Illinois J. Math.}, 59(3):663--674, 2015.

\bibitem[LaV11]{LaVictoire}
P.~LaVictoire.
\newblock Universally {$L^1$}-bad arithmetic sequences.
\newblock {\em J. Anal. Math.}, 113:241--263, 2011.

\bibitem[Les90]{Lesigne90}
E.~Lesigne.
\newblock Un th\'eor\`eme de disjonction de syst\`emes dynamiques et une
  g\'en\'eralisation du th\'eor\`eme ergodique de {W}iener-{W}intner.
\newblock {\em Ergodic Theory Dynam. Systems}, 10(3):513--521, 1990.

\bibitem[Les93]{Lesigne93}
E.~Lesigne.
\newblock Spectre quasi-discret et th\'eor\`eme ergodique de {W}iener-{W}intner
  pour les polyn\^omes.
\newblock {\em Ergodic Theory Dynam. Systems}, 13(4):767--784, 1993.

\bibitem[LM96]{LesigneMauduit-1996}
E.~Lesigne and C.~Mauduit.
\newblock Propri\'et\'es ergodiques des suites {$q$}-multiplicatives.
\newblock {\em Compositio Math.}, 100(2):131--169, 1996.

\bibitem[LMM94]{LesigneMauduitMosse-1994}
E.~Lesigne, C.~Mauduit, and B.~Moss\'e.
\newblock Le th\'eor\`eme ergodique le long d'une suite {$q$}-multiplicative.
\newblock {\em Compositio Math.}, 93(1):49--79, 1994.

\bibitem[LOT99]{LOT}
M.~Lin, J.~Olsen, and A.~Tempelman.
\newblock On modulated ergodic theorems for {D}unford-{S}chwartz operators.
\newblock In {\em Proceedings of the {C}onference on {P}robability, {E}rgodic
  {T}heory, and {A}nalysis ({E}vanston, {IL}, 1997)}, volume~43, pages
  542--567, 1999.

\bibitem[Mau86]{Mauduit-1986}
C.~Mauduit.
\newblock Automates finis et ensembles normaux.
\newblock {\em Ann. Inst. Fourier (Grenoble)}, 36(2):1--25, 1986.

\bibitem[Mau06]{Mauduit-2006}
C.~Mauduit.
\newblock Propri\'et\'es arithm\'etiques des substitutions et automates
  infinis.
\newblock {\em Ann. Inst. Fourier (Grenoble)}, 56(7):2525--2549, 2006.
\newblock Num\'eration, pavages, substitutions.

\bibitem[MMR14]{MartinMauduitRivat-2014}
B.~Martin, C.~Mauduit, and J.~Rivat.
\newblock Th\'eor\`eme des nombres premiers pour les fonctions digitales.
\newblock {\em Acta Arith.}, 165(1):11--45, 2014.

\bibitem[MS98]{MauduitSarkozy-1998}
C.~Mauduit and A.~S\'ark\"ozy.
\newblock On finite pseudorandom binary sequences. {II}. {T}he {C}hampernowne,
  {R}udin-{S}hapiro, and {T}hue-{M}orse sequences, a further construction.
\newblock {\em J. Number Theory}, 73(2):256--276, 1998.

\bibitem[M{\"{u}}l17a]{Mullner-2017+}
C.~M{\"{u}}llner.
\newblock Automatic sequences fulfill the sarnak conjecture.
\newblock {\em Duke Math. J.}, 166(17):3219--3290, 11 2017.

\bibitem[M{\"{u}}l17b]{Mullner-thesis}
C.~M{\"{u}}llner.
\newblock {\em Exponential sum estimates and Fourier analytic methods for
  digitally based dynamical systems}.
\newblock PhD thesis, Technische Universit\:{a}t Wien, 2017.

\bibitem[Pet89]{P89}
K.~Petersen.
\newblock {\em Ergodic theory}, volume~2 of {\em Cambridge Studies in Advanced
  Mathematics}.
\newblock Cambridge University Press, Cambridge, 1989.
\newblock Corrected reprint of the 1983 original.

\bibitem[Que10]{Queffelec-2010}
M.~Queff\'elec.
\newblock {\em Substitution dynamical systems---spectral analysis}, volume 1294
  of {\em Lecture Notes in Mathematics}.
\newblock Springer-Verlag, Berlin, second edition, 2010.

\bibitem[Tao09]{Tao-whatsnew2}
T.~Tao.
\newblock {\em Poincar\'e's {L}egacies, pages from year two of a mathematical
  blog. {P}art {I}}.
\newblock American Mathematical Society, Providence, RI, 2009.

\bibitem[Wal82]{W82}
P.~Walters.
\newblock {\em An {I}ntroduction to {E}rgodic {T}heory}, volume~79 of {\em
  Graduate Texts in Mathematics}.
\newblock Springer-Verlag, New York-Berlin, 1982.

\bibitem[Wie88]{Wierdl88}
M.~Wierdl.
\newblock Pointwise ergodic theorems along the prime numbers.
\newblock {\em Israel J. Math.}, 64(3):315--336, 1988.

\bibitem[WW41]{WienerWintner}
N.~Wiener and A.~Wintner.
\newblock Harmonic analysis and ergodic theory.
\newblock {\em Amer. J. Math.}, 63:415--426, 1941.

\bibitem[ZK15]{Zorin15}
P.~Zorin-Kranich.
\newblock A double return times theorem.
\newblock Preprint, available at https://arxiv.org/abs/1506.05748, 2015.

\end{thebibliography}

\end{document}